\documentclass[reqno, 11pt]{amsart}
\usepackage[margin=2.5cm]{geometry}
\usepackage{mathrsfs}
\usepackage{amsmath}
\usepackage{amsthm}
\usepackage{amsfonts}
\usepackage{amssymb}
\usepackage{url}
\usepackage{enumerate}
\usepackage[pdftex,bookmarks=true]{hyperref}
  \usepackage[usenames, dvipsnames]{color}
\usepackage{verbatim}
\usepackage{pdfsync}
\usepackage{graphicx}
\usepackage{xcolor}
\usepackage{cleveref}
\usepackage{tikz}
\usetikzlibrary{patterns}

\renewcommand\Re{{\operatorname{Re}}}
\renewcommand\Im{{\operatorname{\mathfrak{Im}}}}

\newcommand\R{{\mathbb{R}}}

\renewcommand\P{{\mathbb{P}}}
\newcommand\E{{\mathbb{E}}}

\newcommand\Var{{\operatorname{Var}}}

\newcommand\dist{{\operatorname{dist}}}
\newcommand\Z{{\mathbb{Z}}}

\newcommand\la{\lambda}
\newcommand\1{\mathbf{1}}

\newcommand\By{{\mathbf y}}

\newcommand\BN{{\mathbf N}}

\newcommand\CB{{\mathcal B}}
\newcommand\CC{{\mathcal C}}

\newcommand\CE{{\mathcal E}}

\renewcommand \small {\scriptsize}

\newcommand\eps{\varepsilon}

\newcommand\bs{\backslash}

\newcommand\Cov{{\operatorname{Cov}}}

\renewcommand\1{\mathbf 1}

\newcommand\til{\widetilde}

\newcommand\nc\newcommand
\nc\dmo\DeclareMathOperator
\renewcommand\bs\boldsymbol
\nc\tP{{\til P}}
\nc{\xxi}{\xi}
\nc{\ii}{{\sqrt{-1}}}
\dmo{\Leb}{{Leb}}
\nc{\wt}{\widetilde}
\nc{\mLeb}{m_{\Leb}}
\nc{\tran}{{\mathsf{T}}}
\nc{\mom}{{m}}
\nc{\avec}{{\bs a}}
\nc{\bvec}{{\bs b}}
\nc{\uvec}{{\bs u}}
\nc{\vvec}{{\bs v}}
\nc{\wvec}{{\bs w}}
\nc{\ul}{\underline}
\nc{\us}{{\bs{s}}}
\nc{\bbs}{{\bs{s}}}

\renewcommand\wt{\widetilde}

\newcommand\bxi{\boldsymbol{\xi}}

\theoremstyle{plain}

\newtheorem{theorem}{Theorem}[section]
\newtheorem{conjecture}[theorem]{Conjecture}

\newtheorem{proposition}[theorem]{Proposition}

\newtheorem{lemma}[theorem]{Lemma}
\newtheorem{corollary}[theorem]{Corollary}

\theoremstyle{definition}

\theoremstyle{remark}

\usepackage{xcolor}
\newif\ifshowedits
\showeditsfalse     

\ifshowedits
    \newcommand{\add}[1]{\textcolor{blue}{#1}}
    \newcommand{\del}[1]{\textcolor{red}{\sout{#1}}}
    
    \newcommand{\note}[1]{\textcolor{purple}{\textbf{[Note: #1]}}}
\else
    \newcommand{\add}[1]{#1}
    \newcommand{\del}[1]{}
    
    \newcommand{\note}[1]{}
\fi

\ifshowedits
    \newenvironment{addblock}{\color{blue}}{}
\else
    
\fi

\begin{document}

\title[CLT for Weyl zeros with general coefficients]{Central limit theorem for real zeros of random Weyl polynomials with general coefficients}

\author{Ander Aguirre}

\address{Department of Mathematics\\ University of Wisconsin-Madison\\ 480 Lincoln Dr\\ Madison, WI 53706 USA}
\email{aguirrezarat@wisc.edu}

\author{Hoi H. Nguyen}

\address{Department of Mathematics\\ The Ohio State University \\ 231 W 18th Ave \\ Columbus, OH 43210 USA}
\email{nguyen.1261@osu.edu}
\thanks{H. Nguyen is supported by Simons Travel Grant TSM-00013318.}

\begin{abstract}

For a random polynomial, the number of real zeros \(N_{\mathbb R}\) is a highly nonlinear function of its coefficients, and its statistical properties have been studied extensively. One of the most natural and widely investigated questions is whether \(N_{\mathbb R}\) satisfies a central limit theorem. For various ensembles with iid standard Gaussian coefficients, such central limit theorems have been established in a substantial body of work; see, for instance, \cite{AnL,AADL,ADL,AL,Dal,DNNP,DV,NS}. These results rely on a rich range of tools, including Kac--Rice formulas, moment methods, and Wiener chaos decompositions. In the non-Gaussian setting, however, many of these tools are unavailable. To the best of our knowledge, prior central limit theorems beyond the Gaussian setting were limited to Kac-type polynomials, including hyperbolic polynomials; see the works of Maslova \cite{Mas}, O. Nguyen and Vu \cite{ONgV-CLT}, and, more recently, Do, N. Nguyen, and O'Rourke \cite{DoNgOr}.

In this paper, we prove a central limit theorem for the total number of real zeros of Weyl polynomials whose coefficients are iid copies of a symmetric, mean-zero, variance-one subgaussian random variable \(\xi\). This substantially extends one of the main results of Do and Vu \cite{DV} to a broad class of non-Gaussian distributions, including the Rademacher distribution. Without the symmetry assumption, we prove central limit theorems for the number of real zeros for positive bulk intervals, as well as for \([0,\infty)\). Our proof combines the uniform one-point anti-concentration estimates from our recent work \cite{ANgW-var} with the localization of Weyl polynomials around the coefficient index \(i\approx x^2\). While our proofs use comparison to compute the variances, the CLT deduction is rather direct.



\end{abstract}

\maketitle


\section{Introduction}

\subsection{Introduction}
Over the past several decades, there has been substantial interest in the study of the real zeros of random polynomials and random functions.  A general random polynomial takes the form
\begin{equation}\label{eqn:f:general}
F_n(x)=\sum_{j=0}^n \xi_j p_j(x),
\end{equation}
where \(\xi_j\) are iid copies of a random variable \(\xi\) with mean zero and variance one, and \(p_j(x)\) are deterministic polynomials of degree \(j\). Different choices of the basis \(\{p_j\}\) lead to a variety of important ensembles.

Some of the most classical examples are:
\begin{enumerate}[(i)]
\item {\it Kac polynomials}: \(p_j(x)=x^j\);

\item {\it Hyperbolic polynomials}: $p_j(x) = \sqrt{\frac{L(L+1)\dots (L+j-1)}{j!}}x^j$ for a given $L>0$; 

\item {\it Trigonometric polynomials}: \(p_j(x)=\cos(jx)\), \(p_j(x)=\sin(jx)\), or combinations thereof;

\item {\it Orthogonal polynomial ensembles}: \(\{p_j\}_{j=0}^n\) forms an orthonormal basis with respect to a smooth Borel measure \(\mu\) on \(\R\);

\item {\it Elliptic polynomials}:
\(
p_j(x)=\sqrt{\binom{n}{j}}\,x^j;
\)

\item {\it Weyl polynomials}:
\(
p_j(x)=\frac{1}{\sqrt{j!}}x^j.
\)
\end{enumerate}

Except for the orthogonal polynomial setting (including trigonometric polynomials), all of the above ensembles can be written in the form
\[
F_n=\sum_{j=0}^n a_j\xi_j x^j
\]
for suitable deterministic coefficients \(a_j\).

The zeros and critical points of random functions arise naturally in many areas of mathematics and physics, and have been studied extensively from both probabilistic and analytic perspectives. One of the most important choices for the coefficients \(\xi_j\) is the standard Gaussian distribution, in which case many quantities associated with the zeros admit explicit analytic descriptions.

A fundamental tool in the Gaussian setting is the celebrated Kac--Rice formula. For a Gaussian polynomial \(F_n\), the expected number of real zeros in an interval \(I\subset \R\) can be expressed as
\[
\E N_I=\int_I \rho_1(x)\,dx,
\]
where the first intensity function \(\rho_1\) is given by
\[
\rho_1(x)
=
\frac{1}{\pi}
\sqrt{
\frac{\partial^2}{\partial s\partial t}
\log K(s,t)\Big|_{s=t=x}
},
\]
and
\[
K(s,t)=\E F_n(s)F_n(t)
\]
is the covariance kernel of the process.

More generally, for each \(1\le k\le n\), let
\[
\rho_{k}(x_1,\dots,x_k)
\]
denote the \(k\)-point correlation function of the real zeros of \(F\); see for instance \cite{HKPV}. These functions are characterized by the identity
\[
\E\Bigg[
\sum
\varphi(\zeta_{i_1},\dots,\zeta_{i_k})
\Bigg]
=
\int_{\R^k}
\varphi(x_1,\dots,x_k)
\rho_k(x_1,\dots,x_k)
\,dx_1\cdots dx_k,
\]
valid for every continuous compactly supported test function
\[
\varphi:\R^k\to\R,
\]
where the sum runs over all ordered \(k\)-tuples of distinct real zeros
\(
(\zeta_{i_1},\dots,\zeta_{i_k})
\)
of \(F\).

Heuristically,
\[
\rho_k(x_1,\dots,x_k)
=
\lim_{\eps\to0}
\frac{
\P\big(
\exists \zeta_1\in I_1,\dots,
\exists \zeta_k\in I_k
\big)
}
{\eps^k},
\]
where
\(
I_j=[x_j-\eps/2,x_j+\eps/2].
\)

In principle, the Kac--Rice formula allows one to compute these correlation functions explicitly. More precisely \footnote{Here the subscript $G$ in $\rho_{k,G}$ is used to emphasize that we are working with polynomials of iid standard Gaussian coefficients.},
\[
\rho_{k,G}(x_1,\dots,x_k)
=
\int_{\R^k}
|y_1\cdots y_k|
\,p(\mathbf{0},\By)
\,dy_1\cdots dy_k,
\]
where \(p\) denotes the joint density of the Gaussian vector
\[
(F_n(x_1),\dots,F_n(x_k),
F_n'(x_1),\dots,F_n'(x_k)).
\]
See for instance \cite{BD} for explicit computations in several classical ensembles.

Beyond exact formulas in the Gaussian setting, another central theme in the theory of random polynomials is universality with respect to the distribution of the coefficients. At the global scale, universality phenomena were established by Kabluchko and Zaporozhets \cite{KZ1}. At the local scale, important advances were made by Tao and Vu \cite{TV}, Do--Nguyen--Vu \cite{DONgV1,DONgV2}, Nguyen--Vu \cite{ONgV}, and more recently in \cite{DLNNP,MY}. Collectively, these works imply the following general principle.

\begin{theorem}[Local universality of correlations]\label{theorem:universality}
Assume that the coefficients \(\xi_j\) are iid copies of a random variable \(\xi\) with mean zero, variance one, and bounded \((2+\eps)\)-moment. Then the local correlation functions of the real zeros of the classical ensembles --- including Kac, Weyl, elliptic, and a broad class of orthogonal polynomial ensembles --- are asymptotically the same as in the Gaussian case. In other words
$$\int \varphi (x_{1},\dots, x_{k}) \rho_{k,\bxi}(x_1,\dots, x_k) dx_{1}\dots d x_{k}=(1+o(1))\int \varphi (x_{1},\dots, x_{k}) \rho_{k,G}(x_1,\dots, x_k) dx_{1} \dots dx_{k},$$


for any nice test function $\varphi$, where in some cases $o(1)=n^{-c}$ for some small constant $c$.
\end{theorem}

While these universality results have led to a rather satisfactory understanding of local zero statistics, many important questions remain open. One particularly active direction concerns the fluctuation of the number of real zeros,
\( N_\R\) (or more formally $N_{\R,n}(F_n)$) of $F_n$,
especially the validity of central limit theorems and related asymptotic laws. More precisely, it is natural to conjecture that CLT fluctuation holds for general $\xi$ (see for instance the comments after \cite[Theorem 1.1]{ONgV-CLT}). 
 
\begin{conjecture}\label{conj:CLT} For all models of random polynomials considered above, under the assumption that $\xi$ is subgaussian \footnote{Perhaps the conjecture continues to hold even when $\xi$ has bounded $(2+\eps)$-moment for some given $\eps>0$.} and has mean zero and variance one, we have
$$\frac{N_\R -\E N_\R}{\sqrt{\Var N_\R}} \xrightarrow{d} \BN(0,1).$$
\end{conjecture}

\add{In this paper, we study the Weyl
ensemble.} The number of real zeros of Weyl polynomials has been extensively studied in the literature (see for instance \cite{TV} and the references therein). In particular, when the coefficients are independent standard real Gaussian random variables, it is known that
\[
\E N_\R
=
\left(\frac{2}{\pi}+o(1)\right)\sqrt n,
\qquad
\Var(N_\R)
=
(2K+o(1))\sqrt n,
\]
where \(K>0\) is an explicit constant. 

\add{Meanwhile, the Gaussian CLT for $N_\R$ was first established by Do and Vu \cite[Theorem 4]{DV}}.

\begin{theorem}\label{thm:DV}
Assume that the coefficients \(\xi_i\) are independent standard real Gaussian random variables. Then, as \(n\to\infty\),
\[
\frac{N_\R-\E N_\R}{\sqrt{\Var(N_\R)}}
\xrightarrow{d}
\BN(0,1).
\]
\end{theorem}
Their argument also applies to smooth linear statistics of the real zeros, see \cite[Theorem 5]{DV}. The purpose of the present paper is to establish a non-Gaussian central limit theorem on the positive bulk intervals for which sharp variance asymptotics are available \add{due to \cite{ANgW-var}}. 

\subsection{Our contributions} We will first work with Weyl ensembles where the random coefficients are quite general. Our first main result reads as follows.

\begin{theorem}[main result, CLT on positive bulk intervals]\label{thm:main} Assume that \(\xi\) is subgaussian random variable with mean zero and variance one.
Fix constants \(0<c_1<c_2\) and \(0<\sigma_*<1/2\).  Let $n^{\sigma_*}\le M \le (c_2+1)^{-1} \sqrt{n}$ be a parameter and consider the interval
\[
 I_W=[c_1M,c_2M]\subset [0,\sqrt n-M].
\]
Let \(N_{I_W}\) denote the number of real zeros in \(I_W\) of the Weyl polynomial whose coefficients are iid copies of \(\xi\). Then \[\Var(N_{I_W})\asymp M\] 
and, as $n \to \infty$
\[
\frac{N_{I_W}-\E N_{I_W}}{\sqrt{\Var(N_{I_W})}}
\xrightarrow{d}
\BN(0,1).
\]
\end{theorem}

We are also able to extend the proof to the extended interval $[0,\sqrt{n}]$.

\begin{theorem}[main result, CLT on the extended positive interval]
\label{thm:full-positive-conditional}
Assume that the coefficients of the Weyl polynomial are iid copies of a subgaussian random variable $\xi$ with mean zero and variance one, and let
\[
N_n^+:=N_{[0,\sqrt n]}.
\]
Then we have
\[
\frac{N_n^+-\E N_n^+}
{\sqrt{\Var(N_n^+)}}
\xrightarrow{d}\BN(0,1).
\]
\end{theorem}
In fact, by using the method of Section \ref{section:full-real} to treat the outlier region $(\sqrt{n}, \infty)$ (Subsection \ref{sub:outlier}), one can easily extend Theorem \ref{thm:full-positive-conditional} to the number of real zeros on the entire positive line $[0,\infty)$.

\begin{theorem}\label{thm:full-positive}
The conclusion of Theorem \ref{thm:full-positive-conditional} also holds with
\(N_{[0,\sqrt n]}\) replaced by \(N_{[0,\infty)}\).
\end{theorem}

We next extend our result to the whole real line under the additional assumption that \(\xi\) is symmetric, that is, \(\xi\) has the same distribution as \(-\xi\). Under this assumption, the process \(P_n(-x)\) has the same law as \(P_n(x)\), so the positive-axis anti-concentration estimates and one-sided variance analysis from \cite{ANgW-var} (to be detailed later) transfer directly to the negative axis. We note that even under symmetry, the zero counts on the positive and negative axes are not independent, since they are determined by the same coefficients. As it will be clear later, we handle this dependence by grouping each pair of reflected intervals into a single block.

\begin{theorem}[main result, CLT on the entire real line]
\label{thm:full-real}
Assume that the coefficients are iid copies of a symmetric subgaussian random
variable $\xi$ with mean zero and variance one.  Then
\begin{equation}
\operatorname{Var}(N_{\mathbb R})=(2K+o(1))\sqrt n,
\label{eq:full-real-variance}
\end{equation}
where $K$ is the Gaussian Weyl variance constant from \cite[Theorem 4]{DV},
and
\[
\frac{N_{\mathbb R}-\E N_{\mathbb R}}
{\sqrt{\Var(N_{\mathbb R})}}
\xrightarrow{d}\BN(0,1).
\]
\end{theorem}

Thus, our result applies in particular to Rademacher (Bernoulli) coefficients,
\[
\mathbb P(\xi=\pm1)=\frac12.
\]
We believe that the symmetry assumption can be removed. However, doing so would require extending the (rather lengthy) arguments of \cite{ANgW-var} to the negative axis, where the reflected polynomial has alternating coefficient laws. We leave this extension to future work.

\subsection{Literature on CLT for the Gaussian models}

In the Gaussian setting, there has been substantial recent progress on fluctuation theory and central limit theorems for random polynomials and Gaussian analytic functions; see, for instance, \cite{AnL,AADL,ADL,AL,Dal,DNNP,DV,Mas, NS,ONgV-CLT} and the references therein. The techniques used in these works vary significantly depending on the underlying ensemble.

One robust approach to fluctuation problems is the moment method developed in \cite{AnL}, where the authors study the \(k\)-th moment
\[
\E\left(\frac{N_\R-\E N_\R}{\sqrt{\Var N_\R}}\right)^k
\]
for every fixed \(k\), for various random functions \(F_n\) with smooth correlation kernel \(K(s,t)\). Their analysis relies on clustering properties of the correlation functions \(\rho_k\), a feature that will also appear later. This method has been an important contribution to the study of fluctuations.

However, perhaps the most powerful framework available in the Gaussian setting is the Wiener chaos decomposition, originated from \cite{ADL,AL}, which expresses functionals of Gaussian random variables as orthogonal expansions with respect to the Gaussian measure.

Most relevant to the present paper is the work of Do and Vu \cite{DV} on Weyl polynomials. Their approach is based on a detailed comparison of the infinite Gaussian analytic function
\(
F(x)
=
\sum_{i=0}^{\infty}\frac{\xi_i}{\sqrt{i!}}x^i.
\)
A key ingredient of their argument is to show that truncating this infinite series at degree \(n\) does not significantly affect the local statistics of the real zeros. The infinite model is then analyzed using the cumulant method, together with quantitative estimates on correlation functions.

More precisely, let
\(
\rho_k(x_1,\dots,x_k)
\)
denote the \(k\)-point correlation function of the real zeros. One of the main ingredients in \cite{DV} is a clustering property showing that correlations approximately factorize when two groups of points are sufficiently separated. If \(X_I\) and \(X_J\) are two subconfigurations satisfying
\[
d(X_I,X_J)\ge 2\Delta,
\]
then \cite[Lemma 9]{DV} gives
\[
\left|
\frac{\rho(X)}
{\rho(X_I)\rho(X_J)}
-1
\right|
\le
C_k
\exp\left(-\frac12(d-\Delta)^2\right).
\]
A related estimate, established in \cite[Lemma 11]{DV}, states that
\[
\left|
\rho(X)-\rho(X_I)\rho(X_J)
\right|
\le
C_k
\exp\left(-\frac12(d-\Delta)^2\right).
\]

These estimates ultimately rely on techniques developed in the theory of Gaussian analytic functions, particularly the linear functional approach introduced by Nazarov and Sodin in their influential work on fluctuations of complex zeros. Overall, these studies are designed for Gaussian polynomials, and their methods do not seem to extend readily to non-Gaussian settings, even when the entries \(\xi_j\) have nice non-Gaussian densities.

\subsection{CLT for non-Gaussian ensembles, difficulties, and our approach}

Local universality of zero correlations such as Theorem \ref{theorem:universality} does not by itself provide an error
small enough to control every fixed moment of the centered zero count.  Our
proof instead uses a structural feature special to the Weyl basis: localization
in the coefficient index.

We will work with the normalized polynomial
\begin{equation}\label{eqn:Pn}
P_n(x)=e^{-x^2/2}\sum_{i=0}^n\xi_i\frac{x^i}{\sqrt{i!}}=: \sum_i \xi_i a_i(x),
\end{equation}
which has the same real zeros as the original Weyl polynomial.  For
\(x\asymp M\), the coefficient mass, \add{i.e. $|a_i(x)|^2=e^{-x^2}x^{2i}/i!$}  , (see \cite{ANgW-var}) is exponentially concentrated on
\[
i=x^2+O(x).
\]
After adding a padding of order \(x\sqrt{\log n}\), the discarded tail is
polynomially small in every prescribed power of \(n\), both for the polynomial
and its first derivative.

A small \(C^1\) perturbation, \add{such as discarding a coefficient tail}, 
 can change the number of real zeros only near a
\add{degenerate} zero.  Here we use the uniform two-dimensional
anti-concentration estimate of \cite{ANgW-var}: on the
multiplicative bulk interval \(I_W=[c_1M,c_2M]\), with overwhelming
probability,
\add{\[
\inf_{x\in I_W}\big(|P_n(x)|+|P_n'(x)|\big)>M^{-C}
\]
for an arbitrarily prescribed power $C>0$}.  Therefore
the coefficient tail can be removed without changing the zero count on each
spatial interval. \add{The point of this truncation is to ultimately convert coefficient localization into finite-range dependence}. \add{Indeed}, partition \(I_W\) into intervals \(J_j\) of length
\[
h=M^\gamma,
\]
where \(\gamma>0\) is small.  The localized zero count on \(J_j\) depends
only on coefficients with indices in a window of width
\(O(M\sqrt{\log n})\) around \(\{x^2:x\in J_j\}\).  Since
\[
Mh\gg M\sqrt{\log n},
\]
nonadjacent intervals use disjoint coefficient sets.  Their localized zero counts
are thus exactly independent. We note that the idea of this decomposition is not new, in fact it was already used in \cite{NSV} (and also in \cite{NSV-grad}) in their study of the so-called Jancovici–Lebowitz–Manificat law for large fluctuations of random complex zeros of GAF.

The centered interval counts consequently form, up to an event of arbitrarily
small polynomial probability, a finite-range dependent triangular array.  A
cluster moment expansion then has the standard form: (1) singleton clusters vanish,
(2) clusters of size at least three are negligible, (3) and only pair clusters survive.
The latter reconstruct the variance.  The variance asymptotic
\(\Var(N_{I_W})\asymp M\) is imported from \cite[Theorem 1.12]{ANgW-var},
combined with the known Gaussian Weyl variance.  This yields all Gaussian
moments and hence the central limit theorem. We also note that with  non-Gaussian coefficients, one must look beyond the Kac-Rice formula to obtain these moment bounds. These come from an application of Jensen's formula together with a one-point small-ball estimate.

We close by mentioning several relevant results on central limit theorems and related fluctuation questions for random polynomials with non-Gaussian coefficients. Maslova \cite{Mas}, O. Nguyen and Vu \cite{ONgV-CLT}, and, more recently, Do, N. Nguyen, and O'Rourke \cite{DoNgOr} obtained variance asymptotics and central limit theorems for broad classes of Kac-type and hyperbolic random polynomials, as well as certain derivatives and related extensions. In particular, the latter works develop comparison principles applicable to non-centered random polynomials.

A different approach, based on Wiener chaos decompositions and invariance principles, was developed in \cite{AP} by Angst and Poly for certain smooth statistics related to \(N_{\mathbb R}\) for random trigonometric polynomials. (However, that method does not appear to yield fluctuations of the nonsmooth statistic \(N_{\mathbb R}\) itself for the ensembles considered there). Our work complements these results by establishing a central limit theorem for the number of real zeros of the Weyl ensemble with non-Gaussian iid coefficients.


\subsection{Organization of the paper}

Section \ref{section:smallball} records the uniform one-point
anti-concentration estimate from \cite{ANgW-var} in the form needed here.
In Section \ref{section:CLT} we provide a
finite-range independence decomposition and the moment CLT proving Theorem \ref{thm:main}.
In Section \ref{section:full} we prove Theorem \ref{thm:full-positive-conditional}.
In Section \ref{section:full-real} we will pair reflected intervals into radial blocks,
establish the two-sided variance comparison, control the exterior zeros,
and finally prove Theorem \ref{thm:full-real}.

\medskip

{\bf Notation.}  We write $X =
O(Y)$, $X \ll Y$, or $Y \gg X$ if $|X| \leq CY$
for some absolute constant $C$. The constant $C$ may depend on some parameters, in which case we write
e.g. $X=O_\tau(Y)$ if $C=C(\tau)$.  We write $X \asymp Y$ if $X \gg Y$ and $Y \gg X$.
In what follows, $\|.\|_{\R/\Z}$ is the distance to the nearest integer, and $m=\mLeb(\cdot)$ is the Lebesgue measure. 

We say that an event $\CE_n$ occurs with {\it overwhelming} probability if 
$\P(\CE_n) =1-O(n^{-A})$ for any fixed $A$ (independent of $n$). Here the implied constant is allowed to depend on $A$.

We will assume $n \to \infty$ throughout the note.  On the
exponentially unlikely event that all coefficients vanish, we define every
finite-interval zero count to be zero; this convention has no effect on any
of our conclusions and ensures the deterministic bound $N_I\le n$.

\section{Uniform anti-concentration}
\label{section:smallball}

The only arithmetic input needed in the present proof is a uniform lower
bound preventing a nearly multiple real zero in the bulk.  This was already
proved in \cite{ANgW-var}; we record the consequence in the notation used
here.

Recall that for a scale \(M\), \cite{ANgW-var} works on intervals
\[
I(M)=[c_1M,c_2M]\subset [0,\sqrt n-M],
\]
where \(0<c_1<c_2\) are fixed constants.
Their Theorem 4.4 proves a uniform small-ball estimate for the
two-dimensional random walk \((P_n(x),P_n'(x))\), and Lemma 4.5 includes
the two endpoints. (We note that in \cite[Section 4]{ANgW-var}, $N=M$, $b_i= \sqrt{M} e^{-x^2/2}x^i/\sqrt{i!}$, and $c_i=(b_i(x))'$.)

\begin{proposition}[Uniform non-degeneracy on multiplicative bulk intervals]
\label{prop:ANgW-nondegeneracy}
Fix \(\theta>0\) and \(0<\varepsilon<1/2\).  Let
\[
I=[a,b]=[c_1M,c_2M]
\subset [0,\sqrt n-M]
\]
be an interval of the above form, with \(M\) sufficiently large.  Then
\begin{equation}\label{eqn:1}
\P\left(
\inf_{x\in I}
\sqrt{|P_n(x)|^2+|P_n'(x)|^2}
\le M^{-\theta}
\right)
\ll M^{-\theta+1/2+\varepsilon}.
\end{equation}
Moreover,
\begin{equation}\label{eqn:2}
\P\left(
\min\{|P_n(a)|,|P_n(b)|\}\le M^{-\theta}
\right)
\ll M^{-\theta}.
\end{equation}
\end{proposition}

\begin{proof}
For \(I\subset(0,\infty)\), \eqref{eqn:1} is exactly
\cite[Theorem 4.4]{ANgW-var}, while the endpoint statement \eqref{eqn:2} is the
one-dimensional small-ball estimate \cite[Theorem 4.1]{ANgW-var}; the
combined form is also recorded as \cite[Lemma 4.5]{ANgW-var}.
\end{proof}

We will use the following immediate consequence on the interval of
Theorem \ref{thm:main} (wherein $\sigma_*$ was defined).

\begin{corollary}[Uniform non-degeneracy on \(I_W\)]
\label{thm:smallball:inf}
Fix \(A>0\).  There exists \(B=B(A,\sigma_*)>0\) such that
\begin{equation}
\P\left(
\inf_{x\in I_W}\big(|P_n(x)|+|P_n'(x)|\big)\le n^{-B}
\right)=O(n^{-A}).
\end{equation}
Moreover, for any deterministic collection $\CC$ of at most \(n^{O(1)}\) points
\(a\in I_W\), after increasing \(B\) one also has
\begin{equation}
\P\left(\min_{a \in \CC}|P_n(a)|\le n^{-B}\right)=O(n^{-A}).
\end{equation}
\end{corollary}
We remark that this type of transversality result has been proved to be useful in the study of (various ensembles of) random polynomials, see for instance \cite{DHNgV, NgNgV, ANgW-con}.  
\begin{proof}
Apply Proposition \ref{prop:ANgW-nondegeneracy} with a sufficiently large
fixed \(\theta\).  Since \(M\ge n^{\sigma_*}\),
\[
M^{-\theta+1/2+\varepsilon}
\le n^{-\sigma_*\theta + 1/2+\varepsilon},
\]
which is \(O(n^{-A})\) once \(\theta\) is large enough.  Also
\(M\le\sqrt n\), so \(M^{-\theta}\ge n^{-\theta}\); hence the lower bound
at scale \(M^{-\theta}\) implies the weaker threshold \(n^{-B}\) after
choosing \(B\ge\theta\).

For the deterministic points, use the one-dimensional translated small-ball
estimate \cite[Theorem 4.1]{ANgW-var} and a union bound.  Because the number of
points is only polynomial in \(n\), increasing the exponent in that theorem
makes the total error \(O(n^{-A})\).
\end{proof}

\section{Proof of Theorem \ref{thm:main}: decomposition and moment argument}\label{section:CLT}

Throughout this section \(I_W=[c_1M,c_2M]\) is the interval in
Theorem \ref{thm:main}.  By \cite[Theorem 1.12]{ANgW-var}, together with
the Gaussian Weyl variance asymptotic recalled in \cite{DV} (also recorded in \cite[Theorem 1.6]{ANgW-var}), we have
\begin{equation}\label{eq:var-scale}
 \Var(N_{I_W})\asymp M.
\end{equation}

Fix an integer \(k\ge1\) that later will play the role of {\it moment order} in the proof of Theorem \ref{thm:main}.  Choose
\begin{equation}\label{eqn:gamma}
0<\gamma<\frac{1}{10(3k+1)}
\end{equation}

and put
\[
h=M^\gamma.
\]
Partition \(I_W\) into disjoint consecutive half-open intervals
\[
J_j=[r_j,r_j+h),
\]
with the terminal interval closed at its right endpoint and possibly shorter.
Set
\[
Y_j:=N_{J_j}-\E N_{J_j},
\qquad
Z_M:=\sum_jY_j=N_{I_W}-\E N_{I_W}.
\]

\subsection{Zero-count stability and coefficient localization}

\begin{lemma}[Stability of real zero counts]
\label{lemma:zero-homotopy}
Let \(I=[a,b]\) be compact \note{ why does it have to be mentioned?}and let \(H:[0,1]\times I\to\R\) be \(C^1\).
Assume that, for some \(\eta>0\),
\[
 |H(t,x)|+|\partial_xH(t,x)|\ge\eta
 \qquad(t\in[0,1],\ x\in I),
\]
and
\[
 |H(t,a)|,\ |H(t,b)|\ge\eta
 \qquad(t\in[0,1]).
\]
Then the number of real zeros of \(H(t,\cdot)\) in \(I\) is independent of
\(t\).
\end{lemma}
For our application, we only need to compare the numbers of real zeros corresponding to (t=0) and (t=1). Alternative proofs of the required comparison can be obtained, for instance, from the perturbative approaches developed in \cite{NS-concentration, NgZ, ANgW-con} to study the concentration of the number of real zeros. For the reader's convenience, we provide a proof of Lemma \ref{lemma:zero-homotopy} in Appendix \ref{appendix:lemmas}.

We next need the following analog of \cite[Claim 3.1]{ANgW-var} and \cite[Lemma 6.2]{ANgW-con}, which capture a key property of the Weyl basis: at a spatial point $x$, the coefficient mass of both the polynomial and its derivative is exponentially localized near the index $i=x^2$ on the natural scale $x$.

\begin{lemma}\label{lemma:weyl-C1-tail}
There are constants \(c,C>0\) such that, uniformly for \(x\in I_W\) and
\[
1\le R\le x^{1/3},
\]
we have
\[
\sum_{\substack{0\le i\le n\\ |i-x^2|\ge Rx}}
\left[
\left|e^{-x^2/2}\frac{x^i}{\sqrt{i!}}\right|^2
+
\left|
\left(e^{-x^2/2}\frac{x^i}{\sqrt{i!}}\right)'
\right|^2
\right]
\le
C(1+R^2)e^{-cR^2}.
\]
\end{lemma}
A proof of this result is presented in Appendix \ref{appendix:lemmas} for the reader's convenience.

We now turn to one of the most crucial ingredients of our proof.

\begin{lemma}[coefficient independence decomposition]
\label{lemma:radial-localization}
Fix \(A>0\) and $k \in \Z_+$.  There exists \(K=K(A,k)>0\) such that, with
\[
 R=K\sqrt{\log n},
\]
the following holds with probability \(1-O(n^{-A})\) simultaneously for all
intervals \(J_j\).  Set
\[
W_j:=\Big\{0\le i\le n:\
\dist\big(i,\{x^2:x\in J_j\}\big)\le R(r_j+h)\Big\}
\]
and
\[
P_{n,j}(x):=e^{-x^2/2}\sum_{i\in W_j}\xi_i\frac{x^i}{\sqrt{i!}}.
\]
Then
\[
N_{J_j}(P_n)=N_{J_j}(P_{n,j})
\qquad\text{for every }j.
\]
Moreover, for all sufficiently large \(n\),
\[
W_j\cap W_\ell=\varnothing
\qquad\text{whenever }|j-\ell|\ge2.
\]
\end{lemma}
As already mentioned, a similar (but not identical) independence decomposition was already exploited in \cite[Lemma 5]{NSV} (and also \cite[Section 3]{NSV-grad}).

\begin{proof}
Choose \(B\) large and set \(\delta=n^{-B}\).  By
Corollary \ref{thm:smallball:inf}, after increasing \(B\) if necessary,
with probability \(1-O(n^{-A-10k})\),
\[
\inf_{x\in I_W}\big(|P_n(x)|+|P_n'(x)|\big)>4\delta
\]
and \(|P_n(a)|>4\delta\) at every interval endpoint.

Consider an interval $J_j$. We observe that if \(x\in J_j\) and \(i\notin W_j\), then
$$|i-x^2|\ge R(r_j+h) \ge Rx.$$
Lemma \ref{lemma:weyl-C1-tail} and
Cauchy--Schwarz give, on the event \(\sum_{l=0}^n\xi_l^2\le n^2\),
\[
\sup_{x\in J_j}
\big(|P_n(x)-P_{n,j}(x)|+|P_n'(x)-P_{n,j}'(x)|\big)
\le Cn(1+R)e^{-cR^2/2}.
\]
By Chernoff's bound, the exceptional probability (of \(\sum_{l=0}^n\xi_l^2> n^2\)) is exponentially small in $n$.  By taking \(K\) sufficiently
large, and after taking union bound over all $J_j$, we can make the last equation at most \(\delta\), simultaneously for all intervals.

Next, for \(0\le t\le1\) we let
\[
P_t=(1-t)P_n+tP_{n,j}.
\]
Then on (the closure of) \(J_j\),
\[
|P_t|+|P_t'| \ge |P_{n}|+|P_{n}'|-t |P_{n} - P_{n,j}|-t |P_{n}' - P_{n,j}'|>4\delta - \delta -\delta = 2\delta,
\]
and at the endpoints 
\[|P_t|> |P_{n}|-t |P_{n} - P_{n,j}| > 4\delta - \delta >2\delta.\] 

 Lemma
\ref{lemma:zero-homotopy} applied to $t=0,1$ therefore yields
$$N_{J_j}(P_n)=N_{J_j}(P_{n,j}).$$
For the second statement, the squared intervals corresponding to intervals with \(|j-\ell|\ge2\)
are separated by 
$$\min_j \{(r_j+2h)^2 - (r_j+h)^2 \}\ge 2(c_1M)h \gg Mh$$ 
whereas in the definition of $W_{j}$ we surround each squared interval by a width
\(R(r_j+h) =O(RM)\).  Since \(h=M^\gamma \) is much larger than $\sqrt{\log n}=R/K$ (as $M\ge n^{\sigma_\ast}$), the sets
\(W_j,W_\ell\) are disjoint for large \(n\).
\end{proof}

\subsection{Finite-range independence and local moments}\label{sub:moments}

\begin{proposition}[asymptotic finite-range independence]
\label{prop:radial-independence}
Fix \(p\ge1\) and \(A>0\).  Let \(C_1,\dots,C_s\) be disjoint collections
of interval indices such that
\[
|j-\ell|\ge2
\qquad(j\in C_a,\ \ell\in C_b,\ a\ne b).
\]
For each \(a\), let \(H_a\) be a monomial of total degree at most \(p\) in
\(\{N_{J_j}(P_n):j\in C_a\}\).  Then for sufficiently large $n$
\[
\E(\prod_{a=1}^sH_a)
=
\prod_{a=1}^s\E H_a+O(n^{-A}).
\]
The same statement holds for centered interval counts $Y_j =N_{J_j} - \E N_{J_j}$.
\end{proposition}
\begin{proof}
Roughly speaking, our proof simply relies on the observation that if \(j\in C_a\) and
\(\ell\in C_b\), \(a\ne b\), then \(W_j\cap W_\ell=\varnothing\).
Hence the corresponding families of zero counts of $\{P_{n,j}, j\in C_{a}\}_{a=1}^{s}$ are independent because the coefficients are independent.  Lemma \ref{lemma:radial-localization} then matches the original $P_{n}$
and $P_{n,j}$ zero counts outside an event $\CE$ of arbitrarily small polynomial
probability in $n$. Since every zero count is at most \(n\), the exceptional event
contributes \(O(n^{-A})\) after choosing its exponent sufficiently large.  

More formally, we write
\[
N_j:=N_{J_j}(P_n),
\quad
\widetilde N_j:=N_{J_j}(P_{n,j}).
\]
Thus \(N_j,\widetilde N_j\le n\). For each \(a\), write
\[
H_a=\prod_{j\in C_a}N_j^{m_{a,j}},
\quad
d_a:=\sum_{j\in C_a}m_{a,j}\le p,
\]
and
define the corresponding localized monomial
\[
\widetilde H_a
:=
\prod_{j\in C_a}\widetilde N_j^{m_{a,j}}.
\]
As already observed, the localized monomials are independent because if \(a\ne b\) then
\[
W_j\cap W_\ell=\varnothing
\qquad
(j\in C_a,\ \ell\in C_b)
\]
by Lemma \ref{lemma:radial-localization}. Therefore, the families
\[
\{\widetilde N_j:j\in C_a\},
\qquad a=1,\dots,s,
\]
depend on disjoint collections of coefficients, and hence they are independent. We obtain
\begin{equation}\label{ind:1}
\E(\prod_{a=1}^s\widetilde H_a)
=
\prod_{a=1}^s\E\widetilde H_a.
\end{equation}

It remains to transfer this identity from the localized counts back to
the original counts. Fix \(A'>0\), to be chosen sufficiently large, and
let \(\CE^c\) be the event from Lemma
\ref{lemma:radial-localization} on which
\[
N_j=\widetilde N_j
\quad\text{for every interval }J_j.
\]
Then
\[
\P(\CE)=O(n^{-A'}).
\]
Thus for each \(a\),
\begin{equation}\label{ind:2}
\left|\E H_a-\E\widetilde H_a\right| = \left|\E (H_a \1_{\CE})-\E(\widetilde H_a \1_\CE) +\E (H_a \1_{\CE^c}) - \E(\widetilde H_a \1_{\CE^c})\right|
\le
2n^{d_a}\P(\mathcal E) = O(n^{d_a-A'}).
\end{equation}
More generally, by inserting $\1_\CE$ and $\1_{\CE^c}$ and telescoping, with \(
d=\sum_{a=1}^s d_a\le ps\), we have 
\begin{equation}\label{ind:3}
\left|
\E\Big(\prod_{a=1}^sH_a\Big)
-
\E\Big(\prod_{a=1}^s\widetilde H_a\Big)
\right|
\le
2^s n^d\P(\mathcal E)
=
O(n^{-A'+d}).
\end{equation}
Combining \eqref{ind:1}, \eqref{ind:2} and \eqref{ind:3}, using telescoping again, we obtain
\[
\E\left(\prod_{a=1}^sH_a\right)
=
\prod_{a=1}^s\E H_a+O(n^{-A'+d}).
\]
Choosing \(A'>A+ps\) proves the first assertion.

For centered interval counts, we expand each monomial in the variables
\[
Y_j=N_j-\E N_j
\]
as a finite linear combination of monomials in the uncentered counts
\(N_j\). By applying the first assertion term by term, and increasing \(A'\)
once more if necessary, we complete the proof of the centered version.
\end{proof}

We next provide some crude bounds for local moments.

\begin{lemma}\label{lemma:local-moments}
For every fixed \(q\ge1\), 
\[
\E|Y_j|^q\le n^{o(1)}(1+h^2)^q
\]
uniformly in \(j\), where the implied constant may depend on $q$.
\end{lemma}

\begin{proof}
Let \(x_j\) be the midpoint of \(J_j\).  The interval \(J_j\) is contained
in the complex disk \(B(x_j,2h)\).  For the normalized entire function, writing \(z=x+iy\), Cauchy--Schwarz gives

\begin{align*}
|P_n(z)|^2
&=
\left|e^{-z^2/2}\sum_{i=0}^n
\xi_i\frac{z^i}{\sqrt{i!}}\right|^2\le
\left|e^{-z^2/2}\right|^2
\left(\sum_{i=0}^n\xi_i^2\right)
\left(\sum_{i=0}^n\frac{|z|^{2i}}{i!}\right)\\
&\le
e^{-\Re(z^2)}
\left(\sum_{i=0}^n\xi_i^2\right)e^{|z|^2}
\le
\left(\sum_{i=0}^n\xi_i^2\right)e^{2(\Im z)^2}.
\end{align*}

On \(\sum_i\xi_i^2\le n^2\),
\[
\log\sup_{B(x_j,4h)}|P_n|\le C(\log n+h^2).
\]
By \cite[Theorem 4.1]{ANgW-var}, for every large fixed \(D\),
\[
\P(|P_n(x_j)|\le n^{-D})\le n^{-D'}
\]
with \(D'\) as large as desired after increasing \(D\) (this follows by
writing \(n^{-D}=M^{-C}\) and using \(M\ge n^{\sigma_*}\)).
A union bound over \(O(M/h)\) interval centers and Jensen's formula for complex zeroes count (see for instance \cite[Section 8]{ONgV}) therefore
give, with overwhelming probability,
\[
N_{J_j}\le C_D(\log n+h^2)
\]
simultaneously for all \(j\).  

On the exceptional event we use
\(N_{J_j}\le n\).  The stated moment bound then follows by 
\[ |Y_{j}|^q = |N_{J_j}-\E N_{J_j}|^{q} \le 2^{q-1}(N_{J_j}^q+(\E N_{J_{j}})^q).\]
\end{proof}

\subsection{Cluster expansion} Recall that $k$ is the moment order we will be taking in our proof using the moment method. For a \(k\)-tuple of interval indices \(\mathbf j=(j_1,\dots,j_k)\), we join two
positions whenever their interval indices differ by at most one. We will take the connected
components and call them the {\it clusters} of \(\mathbf j\).

In the first part of our moment computation we single out two sources of negligible contribution. 

\begin{lemma}[singleton clusters]\label{lemma:singleton}
The total contribution to
\[
\E Z_M^k=\sum_{j_1,\dots,j_k}\E Y_{j_1}\cdots Y_{j_k}
\]
from tuples containing a singleton cluster is \(o(M^{k/2})\).
\end{lemma}

\begin{proof}
We apply Proposition \ref{prop:radial-independence} to the separated clusters.
The factor corresponding to a singleton cluster $\{j\}$ is \(\E Y_j=0\).  Hence each such
expectation is \(O(n^{-A})\) for arbitrary \(A\).  For fixed \(k\), the number of $k$-tuples is
\[
O_k\big((M/h)^k\big),
\]
which is polynomial in \(n\).  Since \(A\) is arbitrary, choosing it
sufficiently large makes the total contribution negligible.
\end{proof}

\begin{lemma}[clusters of size at least three]\label{lemma:large-cluster}
The total contribution from tuples having no singleton cluster but at least one
cluster of size at least three is \(o(M^{k/2})\).
\end{lemma}

\begin{proof}
Because there is no singleton and there is one
cluster of size at least three, such a tuple has at most \((k-1)/2\) clusters. This bound is crucial in our estimate.  

Since there are
\(O(M/h)\) interval indices, the number of possible cluster locations is
\[
O_k\big((M/h)^{(k-1)/2}\big),
\]
where for each cluster location there are only \(O_k(1)\) relative choices.

 Let \(\mathcal C_1,\ldots,\mathcal C_s\) denote the clusters of the
tuple \((j_1,\ldots,j_k)\), and put \(m_a=|\mathcal C_a|\). Since
distinct clusters are separated by at least two interval indices,
Proposition \ref{prop:radial-independence} gives, for every fixed \(A>0\),
\[
\E\prod_{\nu=1}^kY_{j_\nu}
=
\prod_{a=1}^s
\E\prod_{\nu\in\mathcal C_a}Y_{j_\nu}
+O(n^{-A}).
\]
For each cluster, H\"older's inequality and Lemma
\ref{lemma:local-moments}, applied to the individual intervals rather
than to their union, give
\[
\begin{aligned}
\left|
\E\prod_{\nu\in\mathcal C_a}Y_{j_\nu}
\right|
&\le
\prod_{\nu\in\mathcal C_a}
\left(\E|Y_{j_\nu}|^{m_a}\right)^{1/m_a}\le
n^{o(1)}(1+h^2)^{m_a}.
\end{aligned}
\]
Consequently, since \(\sum_{a=1}^s m_a=k\),
\[
\left|
\E\prod_{\nu=1}^kY_{j_\nu}
\right|
\le n^{o(1)}(1+h^2)^k+O(n^{-A})
\le n^{o(1)}h^{2k}.
\]
Putting together, the total contribution from tuples having no singleton cluster but at least one
cluster of size at least three is bounded by
\[
(M/h)^{(k-1)/2} n^{o(1)}h^{2k}
=
M^{(k-1)/2+\gamma(3k+1)/2+o(1)}
=o(M^{k/2})
\]
by the choice of \(\gamma\) from \eqref{eqn:gamma}.
\end{proof}
We next turn to the main term, that all the clusters have size two (and hence $k$ is even).

\begin{lemma}[Pair clusters]\label{lemma:pair-clusters}
If $k=2q$, the total contribution of tuples whose clusters all have
size two is
\[
(2q-1)!!\,\Var(Z_M)^q+o(M^q).
\]
If $k$ is odd, no such tuples exist.
\end{lemma}

\begin{proof}
Recall that in the expansion
\[
\E Z_M^{2q}
=
\sum_{i_1,\ldots,i_{2q}}
\E\bigl(Y_{i_1}\cdots Y_{i_{2q}}\bigr),
\]
the indices $i_1,\ldots,i_{2q}$ are the indices of the spatial intervals.
For a pair-cluster tuple, the $2q$ positions
$\{1,\ldots,2q\}$ are partitioned into $q$ pairs according to which
two factors belong to the same cluster.  


Fix a pairing
\[
\pi=\{\{a_1,b_1\},\ldots,\{a_q,b_q\}\}
\]
of $\{1,\ldots,2q\}$.  We first sum over pair-cluster tuples whose
induced pairing is $\pi$.

For such a tuple, write
\[
i_r:=i_{a_r},\qquad j_r:=i_{b_r},
\]
for the interval indices corresponding to the pair $\{a_r,b_r\}$.
Since each cluster has size two, we have
$|i_r-j_r|\le1$, and the $q$ pairs $(i_r,j_r)$ are mutually separated (i.e. they are of distance at least two).
By Proposition \ref{prop:radial-independence},
\[
\E\prod_{r=1}^q Y_{i_r}Y_{j_r}
=
\prod_{r=1}^q \E(Y_{i_r}Y_{j_r})+O(n^{-A})
\]
for arbitrary $A$.  After summing, the total error is negligible.
Hence the contribution associated with $\pi$ equals
\begin{equation}\label{eqn:pi}
\sum_{\substack{|i_r-j_r|\le1\\
 (i_1,j_1),\ldots,(i_q,j_q)\ {\rm mutually\ separated}}}
\prod_{r=1}^q \E(Y_{i_r}Y_{j_r})
+o(M^q).
\end{equation}

Set
\[
S:=\sum_{|i-j|\le1}\E(Y_iY_j).
\]
Without the mutual-separation restriction, the sum in \eqref{eqn:pi} is exactly
$S^q$.  We claim that the terms involving not mutual separation contribute
$o(M^q)$. The idea is similar to the proof of Lemma \ref{lemma:large-cluster}. Indeed, if two of the selected local pairs are adjacent, then the
configuration has at most $q-1$ freely chosen cluster locations.  Since
there are $O(M/h)$ intervals, the number of such choices is
\[
O_q((M/h)^{q-1}).
\]
Furthermore, by Cauchy--Schwarz and Lemma \ref{lemma:local-moments},
\[
|\E(Y_iY_j)|
\le (\E Y_i^2\,\E Y_j^2)^{1/2}
\le n^{o(1)}h^4.
\]
Thus the total contribution of the excluded choices is at most
\[
(M/h)^{q-1} n^{o(1)}h^{4q}
=
n^{o(1)}M^{q-1}h^{3q+1}
=o(M^q),
\]
for our choice of $h=M^\gamma$ with $\gamma>0$ sufficiently small as in \eqref{eqn:gamma}.

Putting together, the contribution corresponding to $\pi$ is
\[
S^q+o(M^q).
\]

Finally, Proposition \ref{prop:radial-independence} also gives
\[
\sum_{|i-j|\ge2}\E(Y_iY_j)=o(M),
\]
and hence
\[
S
=
\sum_{i,j}\E(Y_iY_j)+o(M)
=
\Var(Z_M)+o(M).
\]
Since $\Var(Z_M)=\Theta(M)$,
\[
S^q=\Var(Z_M)^q+o(M^q).
\]
There are $(2q-1)!!$ pairings $\pi$ of the $2q$ positions
$\{1,\ldots,2q\}$, which proves the even case.
The odd case is immediate.
\end{proof}

We now conclude our first main theorem.

\begin{proof}[Proof of Theorem \ref{thm:main}]
Since \(Z_M=N_{I_W}-\E N_{I_W}\), \eqref{eq:var-scale} gives
\[
\Var(Z_M)\asymp M.
\]
For each fixed \(k=2q\) or $k=2q+1$, the preceding three lemmas imply
\[
\E Z_M^{2q}
=(2q-1)!!\,\Var(Z_M)^q+o(M^q)
\]
and
\[
\E Z_M^{2q+1}=o(M^{q+1/2}).
\]
Dividing by the appropriate power of \(\Var(Z_M)\), all moments converge to
the corresponding standard Gaussian moments.  We have thus shown
\[
\frac{N_{I_W}-\E N_{I_W}}{\sqrt{\Var(N_{I_W})}}
\xrightarrow{d}\BN(0,1).
\]
\end{proof}

\section{The enlarged positive interval}
\label{section:full}

We now treat the positive interval \([0,\sqrt n]\).  The
anti-concentration theorem quoted directly from \cite{ANgW-var} is stated on
multiplicative intervals
\[
[c_1M,c_2M]\subset[0,\sqrt n-M].
\]
For the present extension we need the same one-point estimate up to a (small)
polynomial distance from the hard edge.  

Fix
\[
0<\varepsilon_b<\frac14
\]
and define the expanding bulk
\begin{align}
\label{eq:expbulk}B_n:=[n^{\varepsilon_b},\,\sqrt n-n^{\varepsilon_b}].
\end{align}
For convenience, we also set 
$$L_n:=\sqrt n.$$

We will focus our main analysis on this set, while for $B_n^c\cap[0, \sqrt{n}] $ will be easily dealt with via a standard $L^2$-bound (where as usual for any random variable $X$ of bounded second moment, $\|X\|_2 = \sqrt{\E(X^2)}$).

We start with following analog of Corollary \ref{thm:smallball:inf} over $B_n$.

\begin{proposition}[near-edge uniform nondegeneracy]
\label{prop:near-edge}
For every \(A>0\) there exists \(B>0\) such that
\begin{equation}
\P\left(
\inf_{x\in B_n}
\bigl(|P_n(x)|+|P_n'(x)|\bigr)\le n^{-B}
\right)
=O(n^{-A}).
\end{equation}
Moreover, for every deterministic collection of at most \(n^{O(1)}\)
points \(a\in B_n\), after increasing \(B\) if necessary,
\begin{equation}
\label{eq:ne-u-nond}\P\left(\min_a|P_n(a)|\le n^{-B}\right)=O(n^{-A}).
\end{equation}
\end{proposition}

\begin{proof}
See Appendix \ref{appendix:near-edge-smallball} for a more detailed treatment, where we follow the arguments of \cite{ANgW-var} very closely. \add{The only new point relative to \cite{ANgW-var} is that, under \eqref{eq:near-edge-range}, \[n-x^2=(\sqrt n-x)(\sqrt n+x) \ge M^{\varepsilon_g}x \gg Lx\]
for every fixed $L>0$. Consequently, the entire local coefficient-index window used in the Diophantine argument, including its finite-difference shifts, remains below the truncation index $n$; once this is checked, the characteristic-function and Esseen arguments of \cite{ANgW-var} apply unchanged. }
\end{proof}
In what follows we proceed as in Subsection \ref{sub:moments}.
\subsection{Localization and the moment expansion on the expanding bulk}
 We partition \(B_n\) into consecutive
intervals \(J_j\) of length
\[
h=L_n^\gamma,
\]
where \(\gamma>0\) is chosen sufficiently small depending on the fixed
moment order $k$.

The proofs of Lemma \ref{lemma:radial-localization} and Proposition
\ref{prop:radial-independence} remain valid uniformly on \(B_n\).  Indeed,
for \(x\ge n^{\varepsilon_b}\), an interval of length \(h\) has a coefficient
window of width
\[
O(x\sqrt{\log n})
\]
around the squared interval \(\{u^2:u\in J_j\}\), while two nonadjacent
intervals have their squared intervals separated by \(\gg xh\).  Since
\[
h/\sqrt{\log n} \to \infty,
\]
the coefficient windows of nonadjacent intervals are disjoint.  \add{The hypothesis of Proposition }
\ref{prop:near-edge} supplies the uniform \(C^1\)-stability margin needed in Lemma \ref{lemma:zero-homotopy}.

Let
\[
Y_j=N_{J_j}-\E N_{J_j},
\qquad
Z_n:=\sum_jY_j
=
N_{B_n}-\E N_{B_n}.
\]

\note{What do we mean here by conditional on 4.1? That it holds Modulo 4.1}

\begin{proposition}[Gaussian moments on the expanding bulk, conditional on
Proposition \ref{prop:near-edge} ]
\label{prop:expanding-bulk-moments}
For every fixed integer \(q\ge1\),
\begin{equation}\label{eqn:2q}
    \E Z_n^{2q}
=(2q-1)!!\,\Var(Z_n)^q+o(L_n^q),
\end{equation}
and
\begin{equation}\label{eqn:2q+1}
\E Z_n^{2q+1}=o(L_n^{q+1/2}).
\end{equation}
\end{proposition}

\begin{proof}
The proof is exactly the cluster argument from Section \ref{section:CLT},
with the number of intervals now \(O(L_n/h)\), and $L_n$ plays the role of $M$.  The local Jensen estimate
remains uniform:
\[
\E|Y_j|^r\le n^{o(1)}(1+h^2)^r
\]
for every fixed \(r\).  Singleton clusters are negligible by approximate
finite-range independence, \add{i.e. Lemma \ref{lemma:singleton}} .  A tuple having no singleton but a cluster of
size at least three has at most \((k-1)/2\) cluster locations and therefore
contributes \(o(L_n^{k/2})\) after choosing \(\gamma\) sufficiently small. \add{ Indeed with our choice of $h=L_n^{\gamma}$
\[ \left(\frac{L_n}{h}\right)^{(k-1)/2} n^{o(1)}
 h^{2k}
 =
 L_n^{(k-1)/2+\gamma(3k+1)/2+o(1)}
 =
 o(L_n^{k/2}).
\]}

Finally, the pair-cluster argument in Lemma \ref{lemma:pair-clusters}
reconstructs \(\Var(Z_n)^q\), with the factor \((2q-1)!!\) coming from the
pairings of the \(2q\) positions in the moment expansion.

\end{proof}
\note{Just a comment. I think mathematically, this new part has to be stated for the sake of proof completeness. }

\subsection{A variance lower bound}
To establish CLT fluctuation of $Z_n$ by the method of moments, we do not need a precise estimate for the variance of $Z_n$, but it will suffice to show that
\[
\Var(Z_n)\gg L_n.
\]

\begin{lemma}
\label{lemma:expanding-bulk-var-lower}
\add{Suppose we have $B_n$ as defined in \eqref{eq:expbulk}}, then
\begin{equation}\label{eqn:B_n}
\Var(N_{B_n})\gg\sqrt n.
\end{equation}
\end{lemma}

\begin{proof}
Fix a macroscopic interval lying strictly inside the ordinary bulk,
for instance
\[
K_n:=\left[\frac14\sqrt n,\frac12\sqrt n\right].
\]
By \cite[Theorem 1.12]{ANgW-var} and the Gaussian Weyl variance asymptotic,
\begin{equation}\label{eqn:K_n}
\Var(N_{K_n})\asymp\sqrt n.
\end{equation}

Use the same interval partition of \(B_n\).  Let \(A_n\) be the centered sum
of those interval counts whose intervals lie inside \(K_n\) and are at interval
distance at least two from the two endpoints of \(K_n\).  Let \(D_n\) be
the centered sum of all intervals lying outside this core and at interval distance
at least two from it, and let \(E_n\) be the sum of the remaining \(O(1)\)
\add{  boundary intervals  within $K_n$ buffering the core}.

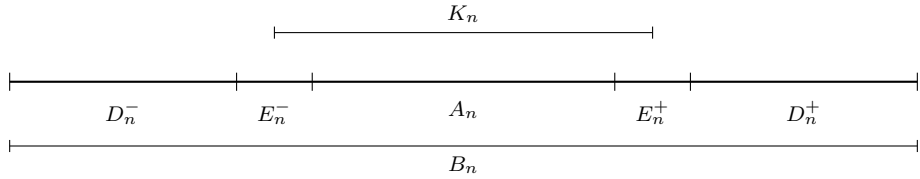
\begin{figure}[ht]
\centering
\begin{tikzpicture}[x=1cm,y=1cm,font=\small]
    \draw[thick] (0,0)--(12,0);

    \foreach \x in {0,3,4,8,9,12}
        \draw (\x,-0.12)--(\x,0.12);

    \node[below] at (1.5,-0.12) {$D_n^{-}$};
    \node[below] at (3.5,-0.12) {$E_n^{-}$};
    \node[below] at (6,-0.12) {$A_n$};
    \node[below] at (8.5,-0.12) {$E_n^{+}$};
    \node[below] at (10.5,-0.12) {$D_n^{+}$};

    \draw (3.5,0.65)--(8.5,0.65);
    \draw (3.5,0.57)--(3.5,0.73);
    \draw (8.5,0.57)--(8.5,0.73);
    \node[above] at (6,0.65) {$K_n$};

    \draw (0,-0.85)--(12,-0.85);
    \draw (0,-0.77)--(0,-0.93);
    \draw (12,-0.77)--(12,-0.93);
    \node[below] at (6,-0.85) {$B_n$};
\end{tikzpicture}
\caption{Decomposition of \(B_n\) into the core \(A_n\), the boundary
parts \(E_n^\pm\), and the exterior parts \(D_n^\pm\).}
\label{fig:variance-core-decomposition}
\end{figure}

By Proposition \ref{prop:radial-independence},
\[
\Cov(A_n,D_n)=O(n^{-A})
\]
for arbitrary \(A\), after summing the $n^{O(1)}$ 
many monomials involved in
the covariance.  Hence
\begin{equation}\label{eqn:AD:A}
\Var(A_n+D_n)
=
\Var(A_n)+\Var(D_n)+o(1)
\ge \Var(A_n)+o(1).
\end{equation}

The difference between \(A_n\) and the centered count on \(K_n\) consists
of only \(O(1)\) boundary intervals.  By Lemma
\ref{lemma:local-moments},
\[
\|A_n-(N_{K_n}-\E N_{K_n})\|_2
\le n^{o(1)}h^2 =o(n^{1/4}),
\]
if we choose \(\gamma<1/8\). Hence \eqref{eqn:K_n} implies
\begin{equation}
\Var(A_n)\asymp\sqrt n.\label{eqn:A_n}
\end{equation}
Similarly,
\begin{equation}\label{eqn:E}
    \|E_n\|_2\le n^{o(1)}h^2=o(n^{1/4}).
\end{equation}

Since
\(
Z_n=A_n+D_n+E_n,
\)
the triangle inequality in \(L^2\), together with \eqref{eqn:AD:A},\eqref{eqn:A_n} and \eqref{eqn:E}, give
\[
\sqrt{\Var(Z_n)}
\ge
\sqrt{\Var(A_n+D_n)}-\|E_n\|_2
\ge
cn^{1/4}-o(n^{1/4}).
\]
This proves \eqref{eqn:B_n}.
\end{proof}

\begin{proposition}[CLT on the expanding bulk]
\label{prop:expanding-bulk-CLT}
\add{Suppose $B_n$ is as defined in \eqref{eq:expbulk} and assume the hypothesis of Theorem \ref{thm:main} for the coefficients, then}
\begin{equation}
  \label{eq:expclt}  \frac{N_{B_n}-\E N_{B_n}}
{\sqrt{\Var(N_{B_n})}}
\xrightarrow{d}\BN(0,1).
\end{equation}
\end{proposition}

\begin{proof}
It follows directly from  Lemma \ref{lemma:expanding-bulk-var-lower},
\[
\Var(Z_n)\gg L_n.
\]
Divide \eqref{eqn:2q}--\eqref{eqn:2q+1} by the appropriate powers of
\(\sqrt{\Var(Z_n)}\).  The error terms tend to zero, so all normalized
moments converge to the standard Gaussian moments, this proves \eqref{eq:expclt} by the moment method.
\end{proof}

\subsection{Putting back the edge intervals}

Let
\[
R_n
:=
N_{[0,n^{\varepsilon_b})}
+
N_{(\sqrt n-n^{\varepsilon_b},\,\sqrt n]}.
\]
We next show the following.

\begin{lemma}[the omitted edge intervals are \(L^2\)-negligible]
\label{lemma:positive-strips}
We have
\begin{equation}
    \E R_n^2=o(\sqrt n),
\qquad
\|R_n-\E R_n\|_2=o(n^{1/4}).
\end{equation}
\end{lemma}

\begin{proof}
The local zero-count estimates for Weyl polynomials from
\cite[Section 12, eq. (88)]{TV} imply that, with overwhelming probability, a
deterministic unit-disk covering of these two intervals contains at most
\(n^{o(1)}\) zeros per disk. \add{We remark that the hypothesis bounded $(2+\delta)$- moment condition of \cite{TV} is satisfied by subgaussianity. }Since \(O(n^{\varepsilon_b})\) disks suffice,
\[
R_n\le n^{\varepsilon_b+o(1)}
\]
with overwhelming probability \add{by the implied union bound}.  \add{On the other hand, we always have the deterministic bound \(R_n\le n\).}  Thus, for a
sufficiently large fixed \(C\),
\[
\E R_n^2
\le
n^{2\varepsilon_b+o(1)}+n^2O(n^{-C})
=
o(n^{1/2}),
\]
because \(\varepsilon_b<1/4\).  The centered estimate follows from
\(\Var(R_n)\le\E R_n^2\).
\end{proof}

We now justify our second main result.

\begin{proof}[Proof of Theorem \ref{thm:full-positive-conditional}]
We have
\[
N_n^+-\E N_n^+
=
Z_n+(R_n-\E R_n).
\]
By Lemma \ref{lemma:expanding-bulk-var-lower},
\[
\sqrt{\Var(Z_n)}\gg n^{1/4},
\]
whereas Lemma \ref{lemma:positive-strips} gives
\[
\|R_n-\E R_n\|_2=o(n^{1/4}).
\]
Hence
\[
\frac{R_n-\E R_n}{\sqrt{\Var(Z_n)}}\longrightarrow0
\]
in \(L^2\).  Moreover the \(L^2\) triangle inequality gives
\[
\left|
\sqrt{\Var(N_n^+)}-\sqrt{\Var(Z_n)}
\right|
\le
\|R_n-\E R_n\|_2
=
o(\sqrt{\Var(Z_n)}),
\]
so
\[
\frac{\Var(N_n^+)}{\Var(Z_n)}\longrightarrow1.
\]
The CLT conclusion for $N_n^+$ now follows from Proposition
\ref{prop:expanding-bulk-CLT} and Slutsky's theorem.
\end{proof}

\section{The entire real line for random symmetric coefficients}
\label{section:full-real}

Throughout this section we impose the additional assumption
\begin{equation}
\xi\stackrel{d}{=}-\xi.
\label{eq:symmetric-atom}
\end{equation}
This assumption preserves the iid model after reflection.  Indeed,
\[
P_n(-x)
=e^{-x^2/2}\sum_{i=0}^n(-1)^i\xi_i\frac{x^i}{\sqrt{i!}}
\stackrel{d}{=}P_n(x)
\]
Note that this
does not make the positive and negative zero counts independent: intervals
near $x$ and $-x$ both use coefficients near the same radial index $i=x^2$.
We therefore put the two reflected intervals into one block.

We keep the same notation $0<\varepsilon_b<1/4$ and $B_n$ from
\eqref{eq:expbulk}, and set
\begin{equation}
\mathcal B_n:=B_n\cup(-B_n)
=\big[-\sqrt n+n^{\varepsilon_b},-n^{\varepsilon_b}\big]
 \cup
 \big[n^{\varepsilon_b},\sqrt n-n^{\varepsilon_b}\big].
\label{eq:two-sided-bulk}
\end{equation}
As before, write $L_n=\sqrt n$.

\subsection{Reflected radial blocks}

We first record the uniform transversality estimate on the reflected bulk.

\begin{lemma}[Two-sided uniform nondegeneracy]
\label{lemma:two-sided-nondegeneracy}
For every $A>0$ there exists $B>0$ such that
\[
\P\left(
\inf_{x\in\mathcal B_n}
\bigl(|P_n(x)|+|P_n'(x)|\bigr)\le n^{-B}
\right)=O(n^{-A}).
\]
The analogous estimate holds simultaneously at any deterministic collection
of at most $n^{O(1)}$ points of $\mathcal B_n$.
\end{lemma}

\begin{proof}
Proposition \ref{prop:near-edge} gives the assertion on $B_n$.  On $-B_n$,
apply the same proposition to the reflected process $x\mapsto P_n(-x)$.
By \eqref{eq:symmetric-atom}, its coefficients $((-1)^i\xi_i)_{i=0}^n$
are again iid copies of $\xi$.  Since
\[
\frac{d}{dx}P_n(-x)=-P_n'(-x),
\]
the value--derivative norm is unchanged.  A union bound proves both claims.
\end{proof}

We next partition $B_n$ into consecutive half-open intervals $J_j$ of length
\[
h=L_n^\gamma,
\]
where $\gamma>0$ will be chosen sufficiently small for each fixed moment
order $k$, and define the reflected block
\[
J_j^\pm:=J_j\cup(-J_j).
\]
Let
\begin{equation}
X_j:=N_{J_j^\pm}-\E N_{J_j^\pm},
\qquad
T_n:=\sum_jX_j=N_{\mathcal B_n}-\E N_{\mathcal B_n}.
\label{eq:paired-block-sum}
\end{equation}
Similarly to Lemma \ref{lemma:radial-localization}, we can establish the following.
\begin{lemma}[Localization of reflected blocks]
\label{lemma:reflected-localization}
Fix $A>0$ and a moment order $k \in \Z_+$.  There is $K=K(A,k)>0$ such that,
with $R=K\sqrt{\log n}$, the following event has probability
$1-O(n^{-A})$.  For every $J_j=[r_j,r_j+h)$, let
\[
W_j:=\left\{0\le i\le n:
\dist\bigl(i,\{x^2:x\in J_j\}\bigr)
\le R(r_j+h)\right\}
\]
and
\[
P_{n,j}(x):=e^{-x^2/2}\sum_{i\in W_j}
\xi_i\frac{x^i}{\sqrt{i!}}.
\]
Then
\[
N_{J_j^\pm}(P_n)=N_{J_j^\pm}(P_{n,j})
\qquad\text{for every }j.
\]
Moreover, $W_j\cap W_\ell=\varnothing$ whenever $|j-\ell|\ge2$.
\end{lemma}

\begin{proof}
For $x<0$, the squared Weyl coefficient and its differentiated square are
\[
\left|e^{-x^2/2}\frac{x^i}{\sqrt{i!}}\right|^2
=e^{-x^2}\frac{|x|^{2i}}{i!},
\qquad
\left|\left(e^{-x^2/2}\frac{x^i}{\sqrt{i!}}\right)'\right|^2
=\frac{(i-x^2)^2}{x^2}e^{-x^2}\frac{|x|^{2i}}{i!}.
\]
Thus Lemma \ref{lemma:weyl-C1-tail} applies with $|x|$ in place of $x$.
On the event in Lemma \ref{lemma:two-sided-nondegeneracy}, the proofs of
Lemma \ref{lemma:radial-localization} and Lemma \ref{lemma:zero-homotopy} are applicable
on both connected components of $J_j^\pm$, using the same coefficient set
$W_j$.  This gives the equality of zero counts.  

The disjointness assertion
is exactly the last part of Lemma \ref{lemma:radial-localization}, since it
depends only on the squared positive intervals.
\end{proof}

The preceding lemma gives approximate finite-range dependence for the paired
counts.  More precisely, if $C_1,\ldots,C_s$ are collections of block indices
separated by distance at least two and $H_a$ is a monomial of fixed degree in
$\{N_{J_j^\pm}:j\in C_a\}$, then
\begin{equation}
\E\prod_{a=1}^sH_a=\prod_{a=1}^s\E H_a+O(n^{-A})
\label{eq:paired-independence}
\end{equation}
for arbitrary $A>0$.  This follows exactly as in Proposition
\ref{prop:radial-independence}, because the localized variables belonging to
different collections depend on disjoint families of the original iid
coefficients.  Centered paired counts satisfy the same conclusion.  Also,
by applying the Jensen argument in Lemma \ref{lemma:local-moments} on the two
components of $J_j^\pm$, for every fixed $q\ge1$,
\begin{equation}
\E|X_j|^q\le n^{o(1)}(1+h^2)^q
\label{eq:paired-local-moments}
\end{equation}
uniformly in $j$.

\subsection{The outlier regions}\label{sub:outlier}

The edge-interval estimate from Lemma \ref{lemma:positive-strips} applies
to the reflected strips by symmetry.  We next control the part beyond the
hard edge, i.e. the zeros of absolute value larger than $\sqrt n$.

\begin{lemma}[outlier zeros]
\label{lemma:exterior-zeros}
For every fixed $\eta>0$, as $n\to \infty$, with overwhelming probability we have
\begin{equation}
N_{\mathbb R\setminus[-\sqrt n,\sqrt n]}
\le n^{\eta+o(1)}.
\label{eq:exterior-whp}
\end{equation}
Consequently, 
\begin{equation}
\E (N_{\mathbb R\setminus[-\sqrt n,\sqrt n]}^2)=o(\sqrt n).
\label{eq:exterior-L2}
\end{equation}
\end{lemma}

\begin{proof}
We again use the results from \cite[Section 12]{TV} for zeros of Weyl polynomial. First, we cover $[-\sqrt{n}-n^{1/2-\eta}, -\sqrt{n}-4, ] \cup [\sqrt{n}+4,\sqrt{n}+n^{1/2-\eta}]$ by $O(\log n)$ dyadic intervals of type
\[[-\sqrt n-2s,-\sqrt n-s] \cup [\sqrt n+s,\sqrt n+2s]
\quad\text{where}\quad 4\le s\le n^{1/2-\eta}.
\]
Note that each such interval can be covered by $O(1)$ disks $B(z_0,r)$ of radius $r$ comparable to $s$ whose slightly
enlarged disks remain outside $B(0,\sqrt n)$. We then apply \cite[Eq. (87)]{TV} to estimate the number of (real and complex and) zeros of $P_n$ in these disks: with overwhelming probability
\begin{equation*}
N_{B(z_0,r)}(P_n)
=
\frac{1}{\pi}
\int_{B(z_0,r)}
\mathbf{1}_{B(0,\sqrt n)}(z)\,dz
+
O\!\left(n^{o(1)}c^{-1}r\right)
+
O\!\left(
\int_{B(z_0,r+c)\setminus B(z_0,r-c)}
\mathbf{1}_{B(0,\sqrt n)}(z)\,dz
\right).
\label{eq:TV-local-circular-law}
\end{equation*}

By the above formula, the main area term and
the boundary-annulus area term vanish, while the remaining error is
$n^{o(1)}$ (we can choose $c,r$ to have order $s$).  Hence each such dyadic disk contains at most $n^{o(1)}$
zeros with overwhelming probability. 

Second, for the interval $[-\sqrt n-4,-\sqrt n] \cup [\sqrt n,\sqrt n+4]$ (and the disks of radius $O(1)$ covering them) we use \cite[Eq. (88)]{TV} (which also follows from \cite[Equation (87)]{TV}) which says that for $C>0$, and for any $z_0 \in B(0,C\sqrt{n})$ and $r\ge 1$, one has with overwhelming probability that
$$N_{B(z_0,r)}(P_n)\ll n^{o(1)}r^2.$$
A union bound over these
$O(\log n)$ dyadic disks thus gives
\begin{equation}
N_{\{\sqrt n<|x|\le\sqrt n+n^{1/2-\eta}\}}
\le n^{o(1)}
\label{eq:near-exterior}
\end{equation}
with overwhelming probability.

Lastly, \cite[Eq. (89)]{TV}, with its parameter
$\varepsilon$ chosen equal to $\eta$, says that all but
$O(n^{\eta+o(1)})$ complex zeros lie in
\[
B\bigl(0,\sqrt n+n^{1/2-\eta}\bigr)
\]
with overwhelming probability.  In particular, the same bound holds for
the real zeros beyond that disk.  Combining this with
\eqref{eq:near-exterior} proves \eqref{eq:exterior-whp}.

For the second moment estimate, we notice that $P_n$ is nonzero with overwhelming probability, and otherwise every zero count under consideration is at most
$n$.  Since overwhelming probability estimates allow an arbitrarily large fixed
power in the exceptional probability,
\[
\E (N_{\mathbb R\setminus[-\sqrt n,\sqrt n]}^2)
\le n^{2\eta+o(1)}+n^2O(n^{-A}).
\]
Choosing $A>3$ and $\eta<1/4$, we obtain \eqref{eq:exterior-L2}.
\end{proof}
Having treated with the outlier range, we now consider it together with the edge intervals. Let
\begin{equation}
D_n:=N_{\mathbb R\setminus\mathcal B_n}.
\label{eq:discarded-full-line}
\end{equation}
We will show the following.
\begin{lemma}
\label{lemma:full-line-discarded}
Under the hypotheses of Theorem \ref{thm:full-real},
\begin{equation}
\E D_n^2=o(n^{1/2}),\quad
\Var(D_n)=o(n^{1/2}).
\label{eq:full-line-discarded-L2}
\end{equation}
In particular, the same conclusion holds for standard Gaussian coefficients.
\end{lemma}

\begin{proof}
Inside $[-\sqrt n,\sqrt n]$, the complement of $\mathcal B_n$ consists of
the central interval $[-n^{\varepsilon_b},n^{\varepsilon_b}]$ and two
hard-edge intervals of length $n^{\varepsilon_b}$. We cover them by
$O(n^{\varepsilon_b})$ unit disks and apply
\cite[Eq. (88)]{TV} again to obtain that, with overwhelming probability, there are at
most $n^{\varepsilon_b+o(1)}$ zeros there. Equation \eqref{eq:exterior-whp} of Lemma
\ref{lemma:exterior-zeros}, with any fixed $\eta<1/4$, handles the exterior.
Since also $\varepsilon_b<1/4$, we have
\[
D_n\le n^{\max\{\varepsilon_b,\eta\}+o(1)}
\]
with overwhelming probability.  Using the deterministic degree bound on
the exceptional event proves the first assertion of
\eqref{eq:full-line-discarded-L2}; the centered assertion then follows as
$\Var(D_n)\le\E D_n^2$. 
\end{proof}
We next turn to the variance of $N_{\CB_n}$.
\subsection{Variance comparison}
Our main goal is the following.

\begin{proposition}\label{prop:two-sided-variance}
Under the hypotheses of Theorem \ref{thm:full-real},
\begin{equation}
\Var(N_{\mathcal B_n})=(2K+o(1))\sqrt n.
\label{eq:two-sided-variance}
\end{equation}
\end{proposition}

Because the proof below compares the variance with that of the Gaussian
ensemble, we distinguish the corresponding zero counts as follows.  For \(I\subset\mathbb R\), let \(N_{I,\boldsymbol{\xi}}\) denote the
number of zeros in \(I\) when the coefficients are iid copies of \(\xi\), and
let \(N_{I,G}\) denote the corresponding number when the coefficients are iid
standard Gaussian random variables.

For Proposition \ref{prop:two-sided-variance} we will need some elementary quantitative consequences of the local universality result, Theorem \ref{theorem:universality}, for Weyl polynomials.  It is the
same estimate used in \cite[Lemma 7.8]{ANgW-var} (whose proof follows \cite{ONgV, TV}), now allowing the unit
intervals to lie on either side of the origin.

\begin{lemma}[local comparison on reflected unit intervals]
\label{lemma:reflected-local-comparison}
There exists $c_0>0$ such that the following holds.  If $U,V$ are intervals
of length at most one contained in $\mathcal B_n$, then, uniformly in $U,V$,
\begin{align}
\left|\E N_{U,\boldsymbol\xi}-\E N_{U,G}\right|&\le n^{-c_0},
\label{eq:local-first-comparison}\\
\left|\E(N_{U,\boldsymbol\xi}N_{V,\boldsymbol\xi})
-\E(N_{U,G}N_{V,G})\right|&\le n^{-c_0}.
\label{eq:local-second-comparison}
\end{align}
Furthermore $\E N_{U,\boldsymbol\xi}, \E N_{U,G}, \E N_{V,\boldsymbol\xi},\E N_{V,G}$ are all of order $O(1)$.
\end{lemma}
\begin{proof}
The proof of \cite[Lemma 7.8]{ANgW-var} derives precisely these estimates
from the quantitative local universality theorem
\cite[Theorem 5.2]{TV} (see also \cite[Theorem 2.6, 5.1]{ONgV}).  That theorem is stated for centers satisfying
\[
n^\varepsilon\le |x_i|\le\sqrt n,
\]
and therefore applies without a change to positive, negative, or reflected
pairs of centers.  For $U=V$, write the second moment as the second factorial
moment plus the first moment and use the one- and two-point versions of local
universality.  The smoothing and repulsion argument in the proof of
\cite[Lemma 7.8]{ANgW-var} is uniform in the allowed centers and yields a
fixed power saving. We can decrease that exponent if necessary to obtain
\eqref{eq:local-first-comparison}--\eqref{eq:local-second-comparison}.
Lastly, the expectation estimates follow from Eq. (107) of \cite[Lemma 7.8]{ANgW-var}. 
\end{proof}
We now proceed to compute the variance.

\begin{proof}(of Proposition \ref{prop:two-sided-variance}) We first choose the block lengths to be
$h=L_n^\gamma$ with $\gamma>0$ so small that
\begin{equation}
\frac{\gamma}{2}<c_0,
\label{eq:gamma-universality}
\end{equation}
where $c_0$ is from Lemma \ref{lemma:reflected-local-comparison}.  For either
coefficient distribution $\xi$ or standard gaussian $G$, \eqref{eq:paired-independence} gives
\[
\sum_{|j-\ell|\ge2}\Cov(N_{J_j^\pm},N_{J_\ell^\pm})=o(L_n),
\]
because there are only polynomially many block pairs and the error exponent
$A$ is arbitrary.

It remains to compare covariances for $|j-\ell|\le1$.  Split each component
of $J_j^\pm$ into $O(h)$ intervals of length at most one.  By Lemma
\ref{lemma:reflected-local-comparison}, the difference between the
$\boldsymbol\xi$ and Gaussian covariance for any pair of these unit
intervals is $O(n^{-c_0})$; here the bounded first moments control the
difference of the products of expectations.  There are $O(L_n/h)$ neighboring
block pairs and $O(h^2)$ unit-interval pairs inside each.  Consequently,
\begin{equation}\label{eqn:xigau}
\Big|\Var(N_{\mathcal B_n,\boldsymbol\xi})
-\Var(N_{\mathcal B_n,G})\Big|
\le O\Big(\frac{L_n}{h}h^2n^{-c_0}\Big)+o(L_n)
=O(L_nh n^{-c_0})+o(L_n)=o(L_n),
\end{equation}
where the last equality follows from $L_n=n^{1/2}$ and
\eqref{eq:gamma-universality}.

For Gaussian coefficients, \cite[Theorem 4]{DV} gives
\[
\Var(N_{\mathbb R,G})=(2K+o(1))L_n.
\]
By Lemma \ref{lemma:full-line-discarded},
\[
\left[\E\Big((N_{\mathbb R,G}-\E N_{\mathbb R,G})
-(N_{\mathcal B_n,G}-\E N_{\mathcal B_n,G})\Big)^2\right]^{1/2}=o(L_n^{1/2}).
\]
The triangle inequality in $L^2$ therefore shows that
\[
\Var(N_{\mathcal B_n,G})=(2K+o(1))L_n.
\]
Combining with \eqref{eqn:xigau}, we obtain \eqref{eq:two-sided-variance}.
\end{proof}

\subsection{Moments and completion of the proof}
We can now establish fluctuation for $N_{\CB_n}$.
\begin{proposition}[CLT on the two-sided interval]
\label{prop:two-sided-bulk-CLT}
Under the hypotheses of Theorem \ref{thm:full-real},
\[
\frac{N_{\mathcal B_n}-\E N_{\mathcal B_n}}
{\sqrt{\Var(N_{\mathcal B_n})}}
\xrightarrow{d}\BN(0,1).
\]
\end{proposition}

\begin{proof}
Fix a moment order $k$ and, in addition to
\eqref{eq:gamma-universality}, choose $\gamma>0$ sufficiently small that
\[
\gamma<\frac{1}{10(3k+1)}.
\]
Expand $T_n^k$ using the paired centered counts $X_j$ from
\eqref{eq:paired-block-sum}, and join two positions whenever their block
indices differ by at most one.  The proofs of Lemmas
\ref{lemma:singleton}--\ref{lemma:pair-clusters} apply verbatim with $Y_j$
replaced by $X_j$, using \eqref{eq:paired-independence} and
\eqref{eq:paired-local-moments}.  For clarity, a tuple with no singleton and
with a cluster of size at least three has at most $(k-1)/2$ freely chosen
block locations, so its total contribution is bounded by
\[
n^{o(1)}\left(\frac{L_n}{h}\right)^{(k-1)/2}h^{2k}
=o(L_n^{k/2}).
\]
For $k=2q$, the pair clusters give
\[
(2q-1)!!\,\Var(T_n)^q+o(L_n^q),
\]
while for odd $k=2q+1$ the moment is $o(L_n^{q+1/2})$.  Proposition
\ref{prop:two-sided-variance} gives $\Var(T_n)\asymp L_n$.  After
normalization, all moments therefore converge to the standard Gaussian
moments.  Moment determinacy proves the proposition.
\end{proof}

\begin{proof}[Proof of Theorem \ref{thm:full-real}]
By \eqref{eq:discarded-full-line},
\[
N_{\mathbb R}-\E N_{\mathbb R}
=T_n+(D_n-\E D_n).
\]
Proposition \ref{prop:two-sided-variance} gives
$\Var(T_n)=(2K+o(1))\sqrt n$, whereas Lemma
\ref{lemma:full-line-discarded} gives
\[
\sqrt{\Var(D_n)}=o(n^{1/4})=o(\sqrt{\Var(T_n)}).
\]
It follows from the $L^2$ triangle inequality that
\[
\left|\sqrt{\Var(N_{\mathbb R})}-\sqrt{\Var(T_n)}\right|
=o(n^{1/4}).
\]
This proves \eqref{eq:full-real-variance}.  Proposition
\ref{prop:two-sided-bulk-CLT} and Slutsky's theorem now give the asserted
central limit theorem.
\end{proof}

\appendix

\section{Near-edge one-point anti-concentration}
\label{appendix:near-edge-smallball}

In this appendix we verify Proposition \ref{prop:near-edge}.  The argument is
the one-point, two-dimensional part of Sections 3--4 of
\cite{ANgW-var}.  We include the verification because the interval in the
main text reaches to within \(n^{\varepsilon_b}\) of the hard edge, whereas
the main bulk interval in \cite{ANgW-var} is stated with a larger edge
separation.  The only modification is to check that all local coefficient
windows used in the Diophantine argument still lie below the truncation
index \(n\).

Fix
\[
0<\varepsilon_b<\frac14
\]
and choose
\begin{equation}
0<\varepsilon_g<2\varepsilon_b.
\label{eq:epsg-choice}
\end{equation}
We first prove the following scale-local statement.

\begin{lemma}[near-edge version of \cite{ANgW-var} on $2d$ small ball bound for a fixed $x$]
\label{lemma:near-edge-pointwise}
Let \(M\to\infty\), let \(x\asymp M\), and assume
\begin{equation}
    0<x\le \sqrt n-M^{\varepsilon_g}.
\label{eq:near-edge-range}
\end{equation}
For every fixed \(C>0\), uniformly in such \(x\),
\begin{equation}
    \sup_{a\in\R^2}
\P\left(
(P_n(x),P_n'(x))\in B(a,M^{-C})
\right)
\ll M^{-2C}.
\label{eq:near-edge-2d-smallball}
\end{equation}
The corresponding one-dimensional estimate
\begin{equation}
\sup_{a\in\R}
\P\left(P_n(x)\in[a-M^{-C},a+M^{-C}]\right)
\ll M^{-C}
\label{eq:near-edge-1d-smallball}
\end{equation}

also holds uniformly.
\end{lemma}

\begin{proof}
We indicate precisely why the proof of
\cite[Theorems 3.13, 3.21, 4.1 and 4.2]{ANgW-var} applies under
\eqref{eq:near-edge-range}.

Set
\[
b_i(x)=\sqrt M\,e^{-x^2/2}\frac{x^i}{\sqrt{i!}},
\qquad
c_i(x)=b_i'(x)
=
\frac{i-x^2}{x}b_i(x).
\]
The local estimates of \cite[Claim 3.1]{ANgW-var} depend only on
\(x\asymp M\) and Stirling's formula, and hence are unchanged.

The covariance input is already stated in the required range:
\cite[Claim 3.2]{ANgW-var} asserts that, for
\(x\asymp M\) and \(x\le\sqrt n-M^{\varepsilon_g}\),
\begin{equation}
\frac1M\sum_{i=0}^n
\begin{pmatrix}b_i(x)\\c_i(x)\end{pmatrix}
\begin{pmatrix}b_i(x)&c_i(x)\end{pmatrix}
=
I_2+\exp(-M^{c})
\label{eq:near-edge-cov}
\end{equation}

for some \(c=c(\varepsilon_g)>0\).

It remains only to check the Diophantine step.  In the proof of
\cite[Theorem 3.13]{ANgW-var}, for a putative vector
\(D=(D_1,D_2)\) one works on an index interval
\begin{equation}
J\subset x^2+[Lx/2,Lx],
\label{eq:ANgW-index-window}
\end{equation}

where \(L>0\) is a fixed constant depending only on the frequency exponent.
Under \eqref{eq:near-edge-range},
\begin{equation}
n-x^2
=
(\sqrt n-x)(\sqrt n+x)
\ge
M^{\varepsilon_g}x.
\label{eq:index-room}
\end{equation}
Since \(M^{\varepsilon_g}\to\infty\), (\ref{eq:index-room}) implies that, for every fixed
\(L\),
\[
x^2+Lx<n
\]
for all sufficiently large \(n\).  Thus the entire index interval
\eqref{eq:ANgW-index-window}, and all the shifted indices \(i+mq\) used in the finite
difference argument of \cite[Subsection 3.12]{ANgW-var}, remain in
\(\{0,\ldots,n\}\).

All subsequent estimates in that subsection are local in this index window
and use only \(x\asymp M\), the Stirling bounds of Claim 3.1, and the facts
\[
|J|\asymp \frac{x}{\log^3M},
\qquad
q\asymp\frac{x}{\log^4M}.
\]
Hence they are unchanged.  For completeness, we mention the only point at
which cancellation between the value and derivative coordinates could
matter.  The leading finite-difference term has the form
\begin{equation}
b_i(z-1)^T
\left(
D_1+xD_2-D_2\frac{i}{x}
\right).
\label{eq:leading-diff}
\end{equation}
The proof of \cite[Theorem 3.13]{ANgW-var} splits into two cases.  If
\[
|D_1+xD_2|\le |D_2(x+L_0)|,
\]
one takes \(L=8L_0\); then \(i/x\ge x+4L_0\) on the chosen part of
\eqref{eq:ANgW-index-window}, and therefore
\begin{equation}
    \left|D_2\frac{i}{x}\right|-|D_1+xD_2|
\ge 3L_0|D_2|.
\label{eq:case1-gap}
\end{equation}
If instead
\[
|D_1+xD_2|>|D_2(x+L_0)|,
\]
one takes \(L=L_0/2\), so that \(i/x\le x+L_0/2\), and
\begin{equation}
|D_1+xD_2|-\left|D_2\frac{i}{x}\right|
\ge \frac{L_0}{2}|D_2|.
\label{eq:case2-gap}
\end{equation}
Thus the leading expression in \eqref{eq:leading-diff} cannot be destroyed by
cancellation.  The error terms are smaller by the same choice of the large
constant \(L_0\), exactly as in \cite[Fact 3.18]{ANgW-var}. It remains to make the bound uniform when $D_2$ is small. Recall that \cite{ANgW-var}  sets $r^2<D_1^2+D_2^2\leq n^A$. If
$|D_2|\geq r/(32L_0)$, then (\ref{eq:case1-gap})
--(\ref{eq:case2-gap}) give
\[
    \left|D_1+xD_2-D_2\frac{i}{x}\right|\geq c r.
\]
Otherwise $D_1^2+D_2^2\geq r^2$ implies $|D_1|\geq 3r/4$; since
$|i/x-x|\leq 8L_0$. In either case,
\[
    \left|D_1-D_2\left(\frac{i}{x}-x\right)\right|
    \geq |D_1|-8L_0|D_2|\geq \frac r2.
\]


Thus the leading expression in (\ref{eq:leading-diff}) is uniformly bounded away from zero.  This proves
the same two-dimensional Diophantine lower bound as
\cite[Theorem 3.13]{ANgW-var}, uniformly under \eqref{eq:near-edge-range}.

Consequently the characteristic-function estimate
\cite[Theorem 3.21]{ANgW-var} holds in this range.  Combining it with
the covariance estimate \eqref{eq:near-edge-cov}, the Esseen argument of
\cite[Theorems 4.1, 4.2]{ANgW-var} is literally unchanged and yields
\eqref{eq:near-edge-2d-smallball}--\eqref{eq:near-edge-1d-smallball}.
\end{proof}

\begin{lemma}[uniform near-edge non-degeneracy]
\label{lemma:near-edge-window}
Fix \(0<c_1<c_2\).  Let
\[
I=[c_1M,c_2M]\subset
[0,\sqrt n-M^{\varepsilon_g}].
\]
For every \(\theta>0\) and every fixed \(0<\eta<1/2\),
\begin{equation}
    \P\left(
\inf_{x\in I}
\sqrt{|P_n(x)|^2+|P_n'(x)|^2}
\le M^{-\theta}
\right)
\ll
M^{-\theta+1+\eta}.
\label{eq:window-uniform}
\end{equation}
Moreover, at each deterministic \(a\in I\),
\begin{equation}
\P(|P_n(a)|\le M^{-\theta})\ll M^{-\theta}.
\label{eq:window-endpoint}
\end{equation}
\end{lemma}

\begin{proof}
The proof of \cite[Theorem 4.4]{ANgW-var} now applies verbatim.
We recall the short argument.  Subgaussianity and the Poisson moment
identities for the Weyl coefficients imply that, outside an event of
probability \(\exp(-M^c)\),
\begin{align}
\sup_{x\in I}
\left\|
\frac{d}{dx}(P_n(x),P_n'(x))
\right\|_2
\le M^{\eta/2}.
\label{eq:derivative-bound}
\end{align}
The estimate \eqref{eq:derivative-bound} remains uniform under \eqref{eq:near-edge-range}. Indeed, the proof of \cite[Equations (69)--(71)]{ANgW-var} uses only upper bounds for the Poisson-weighted sums of the first three derivatives of the Weyl basis; these bounds hold uniformly for $(x\asymp M)$, and truncating the sums at $i=n$ can only improve them, so the same subgaussian concentration and net arguments apply unchanged.
Now partition \(I\) into intervals of length \(M^{-\theta}\).  There are
\(O(M^{1+\theta})\) such intervals.  If the vector
\((P_n,P_n')\) has norm at most \(M^{-\theta}\) somewhere in one of them,
then at its midpoint it has norm
\(O(M^{-\theta+\eta/2})\).  Lemma
\ref{lemma:near-edge-pointwise}, with the exponent adjusted by the fixed
factor \(M^{\eta/2}\), bounds the probability at one midpoint by
\[
O(M^{-2\theta+\eta}).
\]
The union bound therefore gives
\[
O(M^{1+\theta}M^{-2\theta+\eta})
=
O(M^{-\theta+1+\eta}),
\]
which proves \eqref{eq:window-uniform}.  This slightly weaker exponent than the one recorded
in \cite[Theorem 4.4]{ANgW-var} is more than sufficient here because
\(\theta\) may be chosen arbitrarily large.
The endpoint estimate \eqref{eq:window-endpoint} is \eqref{eq:near-edge-1d-smallball}.
\end{proof}

\begin{proof}[Proof of Proposition \ref{prop:near-edge}]
Cover
\[
B_n=[n^{\varepsilon_b},\sqrt n-n^{\varepsilon_b}]
\]
by \(O(\log n)\) overlapping multiplicative intervals
\[
I_\nu=[c_1M_\nu,c_2M_\nu],
\]
with fixed \(0<c_1<c_2\) and \(M_\nu\asymp x\) on \(I_\nu\).
Every \(M_\nu\ge n^{\varepsilon_b}\), while
\(M_\nu\le\sqrt n\).  By \eqref{eq:epsg-choice},
\[
M_\nu^{\varepsilon_g}
\le
n^{\varepsilon_g/2}
=o(n^{\varepsilon_b}).
\]
Hence, for all sufficiently large \(n\),
\[
I_\nu\subset
[0,\sqrt n-M_\nu^{\varepsilon_g}],
\]
so Lemma \ref{lemma:near-edge-window} applies to every \(I_\nu\).

Fix \(A>0\).  Choose \(\theta\) so large that
\[
\varepsilon_b(\theta-1-\eta)>A+10.
\]
Since \(M_\nu\ge n^{\varepsilon_b}\), Lemma
\ref{lemma:near-edge-window} and a union bound over \(O(\log n)\) windows
give
\[
\P\left(
\inf_{x\in B_n}
\sqrt{|P_n(x)|^2+|P_n'(x)|^2}
\le n^{-B}
\right)
=O(n^{-A})
\]
for a suitable fixed \(B>0\).  Since
\(
|u|+|v|\asymp\sqrt{u^2+v^2}.
\)

For \eqref{eq:ne-u-nond}, use the one-dimensional estimate \eqref{eq:window-endpoint}.  If the deterministic
collection contains at most \(n^K\) points, increase \(\theta\) so that
\(\varepsilon_b\theta>A+K+10\), and take a union bound.  This proves the
proposition.
\end{proof}

\section{Proofs of the lemmas}\label{appendix:lemmas}

We first provide an elementary proof of Lemma \ref{lemma:zero-homotopy}.

\begin{proof}
Fix \(t_0\in[0,1]\).  Since
\[
|H(t_0,a)|,\ |H(t_0,b)|\ge\eta,
\]
all zeros of \(H(t_0,\cdot)\) lie in the interior \((a,b)\).  Moreover, at
every such zero \(x_0\),
\[
|\partial_xH(t_0,x_0)|\ge\eta,
\]
so the zero is simple and the zero set contains no accumulation points.  Thus a compact interval contains only finitely many
simple zeros; write them as \(x_1,\ldots,x_m\).  


By the implicit function theorem, for each \(x_r\) there are open neighborhoods
\(U_r\) of \(x_r\) and \(V_r\) of \(t_0\) such that, for every
\(t\in V_r\), the function \(H(t,\cdot)\) has exactly one zero in \(U_r\),
depending continuously on \(t\).  Choose the \(U_r\)'s disjoint.

On the compact complement
\[
K=[a,b]\setminus\bigcup_{r=1}^m U_r
\]
we have \(H(t_0,x)\ne0\).  Hence
\[
\min_{x\in K}|H(t_0,x)|>0.
\]
By continuity, after shrinking the common neighborhood
\(V=\bigcap_rV_r\) of \(t_0\), we have \(H(t,x)\ne0\) for all
\(t\in V\) and \(x\in K\).  Thus, for every \(t\in V\),
\(H(t,\cdot)\) has exactly the same number \(m\) of zeros in \([a,b]\).

Therefore the zero count is locally constant as a function of \(t\).
Since \([0,1]\) is connected, every integer-valued locally constant
function on \([0,1]\) is constant.  Hence the number of real zeros of
\(H(t,\cdot)\) in \(I\) is independent of \(t\).
\end{proof}

We next prove Lemma \ref{lemma:weyl-C1-tail}.

\begin{proof}
Set
\[
a_i(x):=e^{-x^2/2}\frac{x^i}{\sqrt{i!}},
\qquad
\lambda:=x^2,
\]
and let \(X\sim\operatorname{Pois}(\lambda)\) \footnote{Connection to Poisson distribution was also used in \cite[Section 12]{TV} and \cite[Section 4]{ANgW-var}.}.  Then
\[
|a_i(x)|^2
=
e^{-x^2}\frac{x^{2i}}{i!}
=
e^{-\lambda}\frac{\lambda^i}{i!}
=
\P(X=i).
\]
Moreover,
\[
a_i'(x)
=
\left(\frac{i}{x}-x\right)a_i(x)
=
\frac{i-x^2}{x}a_i(x),
\]
and hence
\[
|a_i'(x)|^2
=
\frac{(i-\lambda)^2}{\lambda}\P(X=i).
\]
Thus, after enlarging the sum from \(0\le i\le n\) to all \(i\ge0\), the
left-hand side is bounded by
\begin{align}\label{eq:app2-decomp}
\P\bigl(|X-\lambda|\ge R\sqrt{\lambda}\bigr)
+
\frac1{\lambda}
\E\left[
(X-\lambda)^2
\mathbf 1_{\{|X-\lambda|\ge R\sqrt{\lambda}\}}
\right].
\end{align}

We use the standard Chernoff's bounds for a Poisson random variable; see,
for instance, \cite[Section 2.2]{BLM}. More precisely, if
\(X\sim\operatorname{Poisson}(\lambda)\), then
\begin{equation}\label{Chernoff:upper}
\P(X\ge\lambda+t)
\le
\exp\left\{-\lambda h\left(\frac{t}{\lambda}\right)\right\},
\qquad
h(s):=(1+s)\log(1+s)-s.
\end{equation}
Since for $s\ge 0$
\[
h(s)\ge\frac{s^2}{2(1+s)},
\]
it follows that
\[
\P(X-\lambda\ge t)
\le
\exp\left(-\frac{t^2}{2(\lambda+t)}\right).
\]
For the lower tail, when \(0\le t\le\lambda\),
\begin{equation}\label{Chernoff:lower}
\P(X\le\lambda-t)
\le
\exp\left\{-\lambda h_-\left(\frac{t}{\lambda}\right)\right\},
\end{equation}
where
\[
h_-(s):=(1-s)\log(1-s)+s.
\]
Since for $0 \le s \le 1$,
\[
h_-(s)\ge\frac{s^2}{2}.
\] 
We obtain
\[
\P(\lambda-X\ge t)
\le
\exp\left(-\frac{t^2}{2\lambda}\right).
\]

Consequently, for \(0\le u\le\sqrt{\lambda}\),
\begin{equation}\label{eq:app2-mom1exp1}
\P\bigl(|X-\lambda|\ge u\sqrt{\lambda}\bigr)
\le
2e^{-cu^2}
\end{equation}
for some absolute constant \(c>0\).  Since
\[
R\le x^{1/3}=\lambda^{1/6}<\sqrt{\lambda},
\]
we obtain in particular
\begin{align}
\label{eq:app2-mom1exp2}\P\bigl(|X-\lambda|\ge R\sqrt{\lambda}\bigr)
\le
2e^{-cR^2}.
\end{align}

It remains to estimate the weighted tail.  Put
\[
Y:=\frac{|X-\lambda|}{\sqrt{\lambda}}.
\]
For every nonnegative random variable \(Y\),
\begin{align}\label{eq:app2-2mom}
\E\bigl[Y^2\mathbf 1_{\{Y\ge R\}}\bigr]
=
R^2\P(Y\ge R)
+
\int_R^\infty 2u\,\P(Y\ge u)\,du.
\end{align}
On the range \(R\le u\le\sqrt{\lambda}\), \eqref{eq:app2-mom1exp1} gives
\begin{align}
R^2\P(Y\ge R)
+
\int_R^{\sqrt{\lambda}}2u\,\P(Y\ge u)\,du
&\le
CR^2e^{-cR^2}
+
C\int_R^\infty u e^{-cu^2}\,du \nonumber \\
&\le
C(1+R^2)e^{-cR^2}.
\end{align}
For \(u\ge\sqrt{\lambda}\), we use the upper tail \eqref{Chernoff:upper} for $X \ge \la + u \sqrt{\la} $. Since \(h(s)\ge cs\) for \(s\ge1\), the contribution of
\(u\ge\sqrt{\lambda}\) to the integral in \eqref{eq:app2-2mom} is \(O(e^{-c\lambda})\).
As
\[
R^2\le\lambda^{1/3},
\]
this is bounded by
\(
O(e^{-c\lambda}) = O(e^{-cR^2}).
\)
Therefore
\begin{align}\label{app2-mom2exp}
\E\bigl[Y^2\mathbf 1_{\{Y\ge R\}}\bigr]
\le
C(1+R^2)e^{-cR^2}.
\end{align}
Combining \eqref{eq:app2-decomp}, \eqref{eq:app2-mom1exp2}, and \eqref{app2-mom2exp}, we obtain
\[
\P\bigl(|X-\lambda|\ge R\sqrt{\lambda}\bigr)
+
\E\left[
\frac{(X-\lambda)^2}{\lambda}
\mathbf 1_{\{|X-\lambda|\ge R\sqrt{\lambda}\}}
\right]
\le
C(1+R^2)e^{-cR^2}.
\]
\end{proof}

\end{document}